\documentclass{amsart}
\usepackage{bbm, color, soul}
\usepackage{mathrsfs}
\usepackage{fourier}
\usepackage{enumitem}
\usepackage[foot]{amsaddr}
\usepackage{verbatim, alltt}

\def\E{{\mathbb{E}}}

\def\P{{\mathbb{P}}}

\def\R{{\mathbb{R}}}

\def\zero{{\mathbf{a}}}
\def\set{C}

\def\Id{\text{\normalfont{Id}}}

\newcounter{step}
\newcommand{\newstep}{\noindent \textit{Step~$\thestep$. \addtocounter{step}{1}}}

\providecommand{\Fonction}[5]
{\begin{alignat*}{2}
 & #1 :\; & #2 & \longrightarrow #3 \\
 & & #4 & \longmapsto #5
\end{alignat*}}

\def\Ccal{{\mathcal{C}}}

\def\Ecal{{\mathcal{E}}}
\def\Fcal{{\mathcal{F}}}

\def\Jcal{{\mathcal{J}}}

\def\Ncal{{\mathcal{N}}}

\def\Scal{{\mathcal{S}}}
\def\Tcal{{\mathcal{T}}}

\def\Zcal{{\mathcal{Z}}}

\newcommand{\indicator}[1]{\mathbbm{1}_{\{#1\}}}

\newcommand{\norm}[1]{\lVert #1 \rVert}

\newtheorem{theorem}{Theorem}
\newtheorem{prop}{Proposition}[section]
\newtheorem{lemma}[prop]{Lemma}

\newtheorem*{assumption}{Assumption}
\newtheorem*{notation}{Notation}

\theoremstyle{definition}

\newtheorem*{rem}{Remark}

\title[The weak convergence of regenerative processes using some excursion path decompositions]{The weak convergence of regenerative processes using some excursion path decompositions}

\author[A.\ Lambert]{Amaury Lambert$^1$}
\email{amaury.lambert@upmc.fr}
\address{$^1$Laboratoire de Probabilit\'es et Mod\`eles Al\'eatoires\\
UMR 7599 CNRS and UPMC Univ Paris 06\\
Case courrier 188\\
4 Place Jussieu\\
F-75252 Paris Cedex 05, France}

\author[F.\ Simatos]{Florian Simatos$^2$}
\email{f.simatos@tue.nl}
\address{$^2$Department of Mathematics \& Computer Science \\
Eindhoven University of Technology \\ P.O. Box 513 \\
5600 MB Eindhoven, The Netherlands}

\thanks{While most of this research was carried out, A.\ Lambert was funded by project MANEGE 09-BLAN-0215 of ANR (French national research agency), and F.\ Simatos was affiliated with CWI and sponsored by an NWO-VIDI grant.}

\date{\today}

\begin{document}

\begin{abstract}
We consider regenerative processes with values in some general Polish space. We define their $\varepsilon$-big excursions as excursions $e$ such that $\varphi(e)>\varepsilon$, where $\varphi$ is some given functional on the space of excursions which can be thought of as, e.g., the length or the height of $e$. We establish a general condition that guarantees the convergence of a sequence of regenerative processes involving the convergence of $\varepsilon$-big excursions and of their endpoints, for all $\varepsilon$ in a set whose closure contains $0$. Finally, we provide various sufficient conditions on the excursion measures of this sequence for this general condition to hold and discuss possible generalizations of our approach to processes that can be written as the concatenation of i.i.d.\ motifs.
\end{abstract}

\maketitle

\section{Introduction}

This paper is concerned with the weak convergence of processes that regenerate when hitting some distinguished point, say $a$. These processes are usually called regenerative processes and $a$ regeneration point. When a regenerative process started at $a$ returns to $a$ immediately (e.g., Brownian motion), its excursions are described by a $\sigma$-finite measure, called excursion measure, on the space of c\`adl\`ag paths killed (or stopped) when they return to $a$. Together with some positive parameter, the excursion measure not only describes excursions but actually characterizes the law of the whole process. In this paper, we go one step further and study the extent to which the asymptotic behavior of a sequence of excursion measures contains information about the asymptotic behavior of the entire associated processes. For more details about excursion theory, the reader is referred to Blumenthal~\cite{Blumenthal92:0} for Markov processes or Chapter~$22$ in Kallenberg~\cite{Kallenberg02:0} (and references therein) for the general setting. 
\\

Some special care is needed when dealing with sequences of excursion measures. Indeed, although the framework for the weak convergence of measures is well-studied for finite measures, see for instance Billingsley~\cite{Billingsley99:0}, the technical apparatus available to study $\sigma$-finite measures such as the excursion measures we will be interested in is more limited (see nevertheless~\cite{Kozdron06:0,Lawler07:0,Pardoux11:0} for examples where the convergence of excursion measures is directly dealt with.) For this reason, we will be interested in the weak convergence of probability measures obtained by conditioning the excursion measures. This approach also has a natural sample-path interpretation that can be illustrated by considering a sequence of renormalized random walks converging to Brownian motion. 

Excursions of the random walks away from $0$ converge weakly to the Dirac mass at the trivial excursion constantly equal to $0$. This is a typical behavior due to the fact that the excursion measure of the Brownian motion is infinite. On the other hand, it can be shown that big excursions of the random walks, e.g., excursions with length greater than $\varepsilon > 0$, converge weakly to big excursions of Brownian motion. If $\varepsilon$ can be taken arbitrarily small, one may hope that this convergence characterizes the convergence of the whole process. The main result of the present paper, Theorem~\ref{thm:main} below, provides sufficient conditions for such a statement to hold. It must be noted that the convergence of big excursions does not in general imply tightness (a counter-example is provided at the beginning of Subsection~\ref{sub:tightness}). Our main theorem can therefore be seen as a new way of characterizing accumulation points.

More generally than the length, we can fix a non-negative mapping $\varphi$ on the space of excursions such that the push-forward of the excursion measure by $\varphi$ is still $\sigma$-finite (in the more accurate sense that any half-line $[\varepsilon,+\infty]$ has finite mass), and call $e$ an ``$\varepsilon$-big'' excursion when $\varphi(e)>\varepsilon$. We will sometimes call $\varphi(e)$ the ``size'' of $e$ and say that $e$ is ``measured'' according to $\varphi$. 
The law of an $\varepsilon$-big excursion is then well-defined: it can be equivalently seen as the law of the first $\varepsilon$-big excursion of the process away from $a$, or as the probability measure obtained by conditioning the excursion measure on $\varphi>\varepsilon$.
\\

The approach we develop is initially motivated by queueing theory: in particular, a special case of Theorem~\ref{thm:main} of the present paper was used in Lambert et al.~\cite{Lambert11:0} to study the Processor-Sharing queue. Decomposing a process into its excursions is natural in queueing theory. Indeed, for many stochastic networks the difficulties in analyzing the dynamics arise from the behavior at the boundaries of the state-space. Focusing on excursions that live in the interior of the state-space therefore makes it possible to circumvent this problem. Using Theorem~\ref{thm:main-3} of the present paper, this approach was successfully used in~\cite{Borst:0}
, see Section~\ref{sec:context} for more details.
%
%
%

\subsection*{Organization of the paper} Section~\ref{sec:main} contains the main result of the paper together with the minimal set of notation needed to state it; it also includes in Section~\ref{sec:context} a brief discussion of potential applications. Section~\ref{sec:proof} is devoted to the proof of this main result: Subsection~\ref{sub:notation} sets down notation used throughout the paper and Subsection~\ref{sub:deterministic} contains deterministic continuity results used in Subsection~\ref{sub:proof} to prove Theorem~\ref{thm:main}. Section~\ref{sec:extensions} displays various conditions under which the assumptions in Theorem~\ref{thm:main} hold. Finally, we discuss in Section~\ref{sec:discussion} potential generalizations of our approach to processes which are obtained by the concatenation of i.i.d.\ motifs.

\section{Main result} \label{sec:main}

\subsection{Main result}

Let $V$ be a Polish space with some distinguished element $a \in V$. Let $D$ be the set of c\`adl\`ag functions from $[0,\infty)$ to $V$, endowed with the $J_1$ topology (see next section). For $t\ge 0$, let $\theta_t: D \to D$ be the \emph{shift operator}, defined by $ \theta_t (f) := f(\, \cdot \, + t)$. Let $T: D \to [0,\infty]$ be the first hitting time of $a$, defined by 
$$
T(f) := \inf\{ t > 0: f(t) = a \,\}\in[0,+\infty].
$$ 

Note that we allow $T$ to take the value $+\infty$. We call \emph{excursion} a c\`adl\`ag function stopped whenever hitting $a$, that is, $e\in D$ is an excursion if $e(t) = a$ for every finite $t \geq T(e)$. We will sometimes call $T(e)$ the \emph{length} of $e$. We let $\Ecal \subset D$ denote the set of excursions. Also let $\zero \in \Ecal$ be the function which takes constant value $a$:
\Fonction{\zero}{[0,\infty)}{V}{t}{a.}

Then $\zero$ is the only excursion with null length. For $f \in D$, we call \emph{zero set} of $f$ the set $\Zcal(f):=\{t \geq 0: f(t) = a\}$. From the right-continuity of $f$, one sees that the set $\Zcal^c = [0,\infty) \setminus \Zcal$ is a countable union of disjoint intervals of the form $(g,d)$ or $[g,d)$ called \emph{excursion intervals}, see Kallenberg~\cite[Chapter~$22$]{Kallenberg02:0}. With every such interval, we may associate an excursion $e \in \Ecal$, defined as the function $\theta_g(f)$ stopped at its first hitting time $T(e) = d-g$ of $a$. We call $g$ and $d$ its \emph{left} and \emph{right endpoints}, respectively.

In the rest of the paper, we fix a measurable map $\varphi: \Ecal \to [0,\infty]$ such that
$$
e\not=\zero \Longleftrightarrow \varphi(e) > 0.
$$

Note in particular that $\varphi(\zero) = 0$. We call $\varphi(e)$ the \emph{size} of the excursion $e$. For each $\varepsilon>0$, we say that $e$ is \emph{$\varepsilon$-big}, or just \emph{big} if the context is unambiguous, if its size is strictly larger than $\varepsilon$, and \emph{$\varepsilon$-small} or \emph{small} otherwise. We denote by $D_\varphi \subset D$ the set of c\`adl\`ag functions $f$ such that $\varepsilon$-big excursions are locally finitely many for any $\varepsilon>0$:
$$
	D_\varphi = \Big\{ f\in D: \forall\varepsilon >0, \forall t>0,\text{ the number of $\varepsilon$-big excursions starting before $t$ is finite }\Big\}.
$$

Then we can define $e_\varepsilon(f)$ for $f \in D_\varphi$ as the first excursion $e$ of $f$ satisfying $\varphi(e) > \varepsilon$ and $g_\varepsilon(f)$ as its left endpoint, with the convention $(g_\varepsilon, e_\varepsilon)(f) = (+\infty, \zero)$ if no such excursion exists. The maps $e_\varepsilon$ and $g_\varepsilon$ are measurable maps from $D_\varphi$ to $\Ecal$ and $[0,\infty]$ respectively.
\\

Regeneration of Markov processes is well-studied since It\^o's seminal paper~\cite{Ito72:0}, see for instance Blumenthal~\cite{Blumenthal92:0}. Here, we need not assume that our processes are Markovian. We will say that a process $X$ with law $P$ is regenerative (at $a$) if there exists a measure $P_a$ such that for any stopping time $\tau$
\begin{equation} \label{eq:def-reg}
	P \big( \theta_\tau(X) \in \, \cdot \mid \Fcal_\tau \big) = P_a, \ \text{ $P$-almost surely on } \{ \tau < +\infty, \ X(\tau) = a \},
\end{equation}
where $\Fcal$ is the natural filtration of $X$.
It is known, see Kallenberg \cite{Kallenberg02:0} for rigorous statements, that for any regenerative process, the distributional behavior of its excursions away from $a$ can be characterized by a $\sigma$-finite measure $\Ncal$ on $\Ecal\setminus\{\zero\}$ called \emph{excursion measure}.

Note that if $X$ is a regenerative process with excursion measure $\Ncal$, then $\Ncal(\varphi = 0) = 0$ and $X$ almost surely belongs to $D_\varphi$ if and only if $\Ncal(\varphi > \varepsilon) < +\infty$ for every $\varepsilon>0$. Moreover, if holding times at $a$ are nonzero, then by the regeneration property they must be exponentially distributed. In the rest of the paper (except in the last section) we use the following notation.

\begin{notation}
	In all the paper except in Section~\ref{sec:discussion}, $X_n, X$ are regenerative processes that almost surely belong to $D_\varphi$, $\Ncal$ denotes the excursion measure of $X$ and it is assumed to have infinite mass.
\end{notation}

In the sequel $\overline \set$ for $\set \subset \R$ denotes the closure of $\set$. The following theorem is the main result of the paper.

\begin{theorem} \label{thm:main}
	Let $\set \subset (0,\infty)$ be such that $0 \in \overline \set$ and $\Ncal(\varphi =\varepsilon)=0$ for all $\varepsilon \in \set$. If the sequence $(X_n)$ is tight and for every $\varepsilon \in C$,
	\begin{equation} \label{eq:cond}
		\left( g_\varepsilon, e_\varepsilon, T \circ e_\varepsilon, \varphi \circ e_\varepsilon \right) (X_n) \Rightarrow \left( g_\varepsilon, e_\varepsilon, T \circ e_\varepsilon, \varphi \circ e_\varepsilon \right) (X),
	\end{equation}
	then $X_n \Rightarrow X$.
\end{theorem}

The conditions of the previous theorem are sharp, in the sense that if one of the assumptions is removed, one can build an example where the conclusion does not hold.

The proof of Theorem~\ref{thm:main} essentially relies on continuity properties of some truncation and concatenation operators (Lemmas~\ref{lemma:Phi},~\ref{lemma:continuity-truncation} and~\ref{lemma:limit-truncation}) which allow to identify accumulation points of the tight sequence~$(X_n)$. It must be noted that these operators are not continuous in general (the aforementioned lemmas prove continuity properties under specific assumptions), so that it is not at all natural, and even less obvious, that the previous theorem holds under such minimal assumptions.

Section~\ref{sec:extensions} contains several results related to the assumptions of this theorem. Of particular interest are Subsection~\ref{sub:tightness}, where we discuss the implications of the assumption~\eqref{eq:cond} with respect to tightness of the sequence $(X_n)$ (we show that~\eqref{eq:cond} does not imply tightness but simplifies its proof), and Subsection~\ref{sub:excursion-measures}, where we consider the special case where the $X_n$'s have finite excursion measures and give various conditions on their excursion measures that imply tightness or partial versions of~\eqref{eq:cond}.

Finally, we discuss in Section~\ref{sec:discussion} the possibility to extend Theorem~\ref{thm:main} to processes which can be written as the concatenation of i.i.d.\ paths, which we call \emph{motifs}, even if these motifs are not excursions.

\subsection{Potential applications} \label{sec:context}

Before delving into technical details, we discuss in this section the potential of applications of the main results of the present paper. As mentioned in the introduction, a natural context in which these ideas can be applied is queueing theory. In this context, the results of the present paper have been used in~\cite{Lambert11:0} to improve and simplify results previously known and in~\cite{Borst:0, Lambert12:1} to solve two open problems.
%

The initial motivation for this work comes from the study of the Processor-Sharing queue, a fundamental service discipline which is notoriously challenging to analyze. The scaling limit of the Processor-Sharing queue, assuming finite fourth moment of the service distribution, has been derived by Gromoll~\cite{Gromoll04:0}, by building on~\cite{Gromoll02:0} and using the state-space collapse approach. It is fair to say that this method, although very intuitive and powerful, is also extremely technical to implement. In Lambert et al.~\cite{Lambert11:0}, we revisited this result with the approach of the present paper: in addition to a much shorter and less technical proof, we could also improve the moment condition on the service distribution to a minimal finite variance assumption (however, by considering general arrival processes and measure-valued processes, Gromoll's result~\cite{Gromoll04:0} is on some other aspects stronger than ours, see~\cite{Lambert11:0} for a detailed discussion). In~\cite{Lambert11:0} a special case of Theorem~\ref{thm:main} was proved and used: indeed, the proof of Theorem~\ref{thm:main} can be simplified when the sequence $(X_n)$ is known to be C-tight.

Moreover, the results of the present paper allowed to solve a long-standing open question regarding the scaling limit of the Processor-Sharing queue in the case of service distribution with infinite second moment. In particular, in~\cite{Lambert12:1} uniqueness of the scaling limit was proved, and the only possible limit was characterized through its excursion measure; in this context, the results of Section~\ref{sub:excursion-measures}, in particular Theorem~\ref{thm:ex}, were used. The example of the Processor-Sharing queue is also interesting because it shows the added-value of allowing for a general function $\varphi$: because of time-change manipulations, it is natural in this context to measure excursions according to $\varphi(e) = \int_0^{T(e)} (1/e)$. As a last example of the potential usefulness of our approach in queueing theory, let us mention the recent paper~\cite{Borst:0}: there, Theorem~\ref{thm:main-3} of the present paper was invoked to derive the scaling limit of a multi-dimensional stochastic network. Both for the Processor-Sharing queue in the infinite variance case and for this last example, it is not clear whether other classical approaches (e.g., convergence of the finite-dimensional distributions, continuous mapping theorem, martingale functional central limit theorem) could be used.

\section{Proof} \label{sec:proof}

\subsection{$J_1$ topology} \label{sub:notation}
Let $v: V^2 \to [0,\infty)$ be the metric on $V$ and write $v(x) = v(x,a)$ for $x \in V$. For $f \in D$ and $t \geq 0$ set
$$
\Delta f(t) = v(f(t), f(t-))
$$
with the convention $f(t-) = f(t)$ for $t = 0$. If $f, h \in D$ and $m \geq 0$, let $v(f,h) \in D$ denote the function
\Fonction{v(f,h)}{[0,\infty)}{[0,\infty)}{t}{v(f(t), h(t))}
and let $v_m(f,h), v_\infty(f,h) \geq 0$ be the numbers
\[ v_m(f,h) := \sup_{[0, m]} v(f, h) \ \text{ and }\ v_\infty(f,h) := \sup_{[0, \infty)} v(f, h) = \lim_{m \to +\infty} v_m(f,h). \]
For any $f\in D$ and $m\in[0,\infty)$, we define 
$$
v(f) := v(f,\zero)\quad\mbox{ and }\quad v_m(f) = v_m(f, \zero).
$$

Let $\Lambda$ be the set of real-valued functions which are continuous, strictly increasing, unbounded and start at $0$:
$$\Lambda = \big\{ \lambda: [0,\infty) \to [0,\infty): \lambda \text{ is an increasing bijection}\big\}.$$

If $f: [0,\infty) \to \R$ is a real-valued function and $m \geq 0$, set $\norm{f}_m := \sup_{[0,m]} |f|$ and $\norm{f}_\infty := \sup_{[0,\infty)} |f|$. Let $\Id \in \Lambda$ be the identity map and $\Lambda_m$ for $m \in [0,\infty]$ be the set of $\Lambda$-valued sequences $(\lambda_n)$ such that $\norm{\lambda_n - \Id}_m \to 0$. For $f_n, f \in D$ we write $f_n \to f$ for the convergence in the $J_1$ topology. We will indifferently use two equivalent characterizations of this convergence: either that there exists a $\Lambda$-valued sequence $(\lambda_n)$ such that $(\lambda_n) \in \Lambda_m$ and $v_m( f_n \circ \lambda_n, f) \to 0$ for every $m$ in a discrete, unbounded set, see for instance Stone~\cite{Stone63:1}; or that there exists a sequence $(\lambda_n) \in \Lambda_\infty$ such that $v_m( f_n \circ \lambda_n, f) \to 0$ for every $m$ in a discrete, unbounded set, see for instance Billingsley~\cite{Billingsley99:0}.

If $f_n$ and $f$ are real-valued c\`adl\`ag functions, we will also use the notation $f_n \to f$ to denote convergence in the corresponding $J_1$ topology. Note that when $f_n, f \in D$ are such that $f_n \to f$, then $v(f_n) \to v(f)$, as can be seen from the following inequality, valid for any $f_n, f \in D$, $\lambda_n \in \Lambda$ and $m \geq 0$ and that results from the reversed triangular inequality:
\begin{multline*}
	\norm{v(f_n) \circ \lambda_n - v(f)}_m = \sup_{0 \leq t \leq m} \left| v\left( f_n(\lambda_n(t)), a \right) - v(f(t), a) \right|\\
	\leq \sup_{0 \leq t \leq m} v\left( f_n(\lambda_n(t)), f(t) \right) = v_m(f \circ \lambda_n, f).
\end{multline*}

So far, we have defined the convergence of real-valued and $V$-valued c\`adl\`ag functions in the corresponding $J_1$ topologies. In the following, (finite or infinite) product spaces will always be equipped with the product topology. For instance, if for each $n \geq 1$ we have a sequence $U_n = (u_{n,k}, k \geq 1)$ where $u_{n,k}$ for each $k$ is either a real number or an $\R$- or $V$-valued function, then we note $U_n \to U = (u_k, k \geq 1)$ to mean that $u_{n,k} \to u_k$ for each $k \geq 1$.

\begin{rem}
	If $f,h \in D$, then $(f,h)$ can be seen as an element of $D \times D$ or as an element of $D(V^2)$, the space of c\`adl\`ag functions taking values in $V \times V$. As topological spaces, for the $J_1$ topology, the topology of $D(V^2)$ is strictly finer than the product topology of $D \times D$ (see Jacod and Shiryaev~\cite[Remark~VI.$1.21$]{Jacod03:0}). In particular, if $(f_n)$ and $(h_n)$ are relatively compact sequences of $D$, then the sequence $(f_n, h_n)$ is also relatively compact in the product topology (which we consider in the present paper) but may fail to be so in the topology of $D(V^2)$.
\end{rem}

\subsection{Truncation and path decompositions} \label{sub:deterministic}

The proof of Theorem~\ref{thm:main} essentially relies on continuity properties of the family $(\Phi_\varepsilon, \varepsilon > 0)$ of truncation maps, defined as follows: let $\varepsilon > 0$, $f \in D$, $t \geq 0$ and $e^{st}(f,t) \in \Ecal$ be the excursion of $f$ straddling $t$ (which is unambiguously defined for $t \in [0,\infty) \setminus \Zcal(f)$, while for $t \in \Zcal(f)$ is defined as $e^{st}(f,t) = \zero$). Then $\Phi_\varepsilon(f)(t)$ is defined by:
\begin{equation} \label{eq:Phi}
	\Phi_\varepsilon(f)(t) = \begin{cases}
	f(t) & \text{ if } \varphi(e^{st}(f,t)) > \varepsilon,\\
	a & \text{ else.}
\end{cases}
\end{equation}

In words, $\Phi_\varepsilon$ is the map that truncates $\varepsilon$-small excursions to $\zero$. We have the following intuitive result, pertained to continuity properties of the family $(\Phi_\varepsilon, \varepsilon > 0)$ as $\varepsilon \to 0$.

\begin{lemma} \label{lemma:Phi}
	For any $f \in D$, we have $\Phi_\varepsilon(f) \to f$ as $\varepsilon \to 0$.
\end{lemma}

\begin{proof}
	Let $m \geq 0$: then by definition for any $t \geq 0$ we have $v(\Phi_\varepsilon(f)(t), f(t)) = v(f(t))$ if $\varphi(e^{st}(f,t)) \leq \varepsilon$ and $0$ otherwise, so that
	\[ v_m \left( \Phi_\varepsilon(f), f \right) \leq \sup \left\{ v_m(e) : e \in \Tcal_\varepsilon \right\} \]
	where $\Tcal_\varepsilon$ is the set of excursions $e$ of $f$ with $\varphi(e) \leq \varepsilon$ and left endpoint $g \leq m$. Let $\delta > 0$: because $f$ is c\`adl\`ag, there may only be finitely many excursions $e$ of $f$ starting before $m$ with $v_m(e) \geq \delta$. Let $\varepsilon_0 > 0$ be the smallest size of these excursions: then clearly, for $\varepsilon < \varepsilon_0$ any excursion $e \in \Tcal_\varepsilon$ must satisfy $v_m(e) < \delta$, in particular $\sup \{ v_m(e): e \in \Tcal_\varepsilon \} \leq \delta$ for $\varepsilon < \varepsilon_0$ which proves the result.
\end{proof}

We are now interested in continuity properties of $\Phi_\varepsilon$ for a given $\varepsilon > 0$. We adopt a more general viewpoint, inspired by Lemma~$4.3$ in Whitt~\cite{Whitt05:0}, and see $\Phi_\varepsilon$ as the concatenation of paths according to a given subdivision. To formalize this idea we introduce some additional notation.
\\

We call $S = (s_k, k \geq 0) \in [0,\infty]^\infty$ a \emph{subdivision} if the two following conditions are met: $(1)$ $s_k \leq s_{k+1}$ for $k < |S|$ and $s_k = +\infty$ for $k \geq |S|$, with $|S| = \inf\{ k \geq 0: s_k = +\infty \} \in \{1, \ldots, +\infty\}$ and $(2)$ $\bigcup_{k \geq 0} [s_k, s_{k+1}) = [0,\infty)$, where from now on we adopt the convention $[+\infty, +\infty) = \emptyset$. Note in particular that these two conditions imply that $s_0 = 0$.

The subdivision $S$ is called \emph{strict} if $s_k < s_{k+1}$ for $k < |S|$ and in the rest of the paper, $\Scal$ and $\Scal_+$ denote respectively the set of subdivisions and of strict subdivisions.
\\

For $S = (s_k, k \geq 0) \in \Scal$ we consider the truncation and patching operators $\Phi^S: D \to D$ and $\Psi^S: D^\infty \to D$. They are defined as follows, for $f \in D$, $\zeta = (\zeta_k, k \geq 1) \in D^\infty$ and $t \geq 0$:
\[ \Phi^S(f)(t) = \begin{cases}
	f(t) & \text{ if } t \in \bigcup_{k \geq 0} [s_{2k+1}, s_{2k+2}),\\
	a & \text{ else}
\end{cases} \]
and
\[ \Psi^S(\zeta)(t) = \begin{cases}
	\zeta_{k+1}(t-s_{2k+1}) & \text{ if there exists $k \geq 0$ such that } t \in [s_{2k+1}, s_{2k+2}),\\
	a & \text{ else.}
\end{cases} \]

Note that if it exists, the $k$ appearing in the definition of $\Psi^S$ is necessarily unique. For $S \in \Scal$ we also define the decomposition operator $E^S: D \to D^\infty$ for $f \in D$, $t \geq 0$ and $k \geq 1$ by $E^S(f) = (e_k^S(f), k \geq 1) \in D^\infty$ with
\[ e_k^S(f)(t) = \begin{cases}
f(t + s_{2k-1}) & \text{ if } 0 \leq t < s_{2k} - s_{2k-1},\\
a & \text{ else} \end{cases} \]
with the convention in the above display that $s_{2k} - s_{2k-1} = 0$ if $s_{2k-1} = +\infty$. In particular, if $|S| < +\infty$ then only finitely many of the $e^S_k$'s may be different from $\zero$.
\\

In words, $\Phi^S(f)$ is obtained from $f$ by truncating it to $a$ on odd intervals (first interval, third interval, \ldots) of $S$; $e_k^S(f)$ is the path taken by $f$ on the $k$-th even interval; $\Psi^S(\zeta)$ is obtained from the sequence of paths $\zeta$ by placing the $n$-th one on the $n$-th even interval (and truncating it to $a$ away from this interval). It is therefore plain that these three operators are related as follows.

\begin{lemma} \label{lemma:relation-operators}
	For any $S \in \Scal$, we have $\Phi^S = \Psi^{S} \circ E^S$.
\end{lemma}

Lemmas~\ref{lemma:continuity-truncation} and~\ref{lemma:limit-truncation} are the most important continuity results on these operators.

\begin{lemma} [Continuity of the truncation map] \label{lemma:continuity-truncation}
	Let $S_n, S \in \Scal$ and $f_n, f \in D$. If $S_n \to S$, $f_n \to f$ and the sequence $(\Phi^{S_n}(f_n))$ is relatively compact, then $\Phi^{S_n}(f_n) \to \Phi^S(f)$.
\end{lemma}

\begin{proof}
	Write $S = (s_k)$ and $S_n = (s_{n,k})$, let $\phi$ be any accumulation point of the relatively compact sequence $(\Phi^{S_n}(f_n))$ and assume without loss of generality that $\Phi^{S_n}(f_n) \to \phi$. Consider $t \geq 0$ such that $t \notin S$ and both $\phi$ and $f$ are continuous at $t$; in particular $\Phi^{S_n}(f_n)(t) \to \phi(t)$ and $f_n(t) \to f(t)$. Let $k \geq 0$ such that $s_{k} < t < s_{k+1}$ and consider $n$ large enough such that $s_{n,k} < t < s_{n,k+1}$ ($S_n \to S$ implies that $s_{n,k} \to s_k$ for every $k$). If $k$ is even then we have $\Phi^{S_n}(f_n)(t) = a = \Phi^S(f)(t)$ while if $k$ is odd then we have $\Phi^{S_n}(f_n)(t) = f_n(t) \to f(t) = \Phi^S(f)(t)$. This proves $\Phi^{S_n}(f_n)(t) \to \Phi^S(f)(t)$ and since $\Phi^{S_n}(f_n)(t) \to \phi(t)$ this shows that $\phi$ and $\Phi^S(f)$ coincide on a dense subset of $[0,\infty)$. Since they are c\`adl\`ag they must be equal everywhere, hence the result.
\end{proof}

\begin{rem}
	The sequence $(\Phi^{S_n}(f_n))$ is not necessarily relatively compact under the assumptions of the previous lemma, as can be seen from the example $f(t) = 1 + \indicator{1 \leq t < 2}$, $f_n(t) = 1 + \indicator{1+1/n \leq t < 2-1/n}$ and $S_n = S = (0,1,2,+\infty,\ldots)$ (where $a = 0$).
\end{rem}

We now want to prove Lemma~\ref{lemma:limit-truncation}, which considers the convergence of the sequence $(\Phi^{S_n}(f_n))$ when the two sequences $(S_n)$ and $E^{S_n}(f_n)$ converge. To this end we need a preliminary result on the concatenation map $\Ccal: D \times [0,\infty) \times D \to D$ defined for $f, h \in D$ and $s, t \geq 0$ by
\[ \Ccal(f,t,h)(s) = \begin{cases}
	f(s) & \text{ if } s < t,\\
	h(s-t) & \text{ if } s \geq t.
\end{cases} \]

\begin{lemma} [Continuity of the concatenation map] \label{lemma:continuity-concatenation}
	If $f_n, h_n, f, h \in D$ and $t_n, t > 0$ are such that $t_n \to t$, $f_n \to f$, $h_n \to h$ and $\Delta f_n(t_n) \to \Delta f(t)$, then $\Ccal(f_n, t_n, h_n) \to \Ccal(f,t,h)$.
\end{lemma}

\begin{proof}
	Let $(\lambda_n), (\mu_n) \in \Lambda_\infty$ such that $v_m(f_n \circ \lambda_n, f) \vee v_m(h_n \circ \mu_n, h) \to 0$ for every $m \in M$ with $M$ some discrete, unbounded subset of $[0,\infty)$. If $f$ is discontinuous at $t$, then the assumption $\Delta f_n(t_n) \to \Delta f(t)$ implies that $\lambda_n(t) = t_n$ for $n$ large enough, see Proposition~VI.$2.1$ in Jacod and Shiryaev~\cite{Jacod03:0}. Else, $f$ is continuous at $t$ and by modifying $\lambda_n$ locally we can assume without loss of generality that $\lambda_n(t) = t_n$. In either case, we can assume without loss of generality that $\lambda_n(t) = t_n$ for every $n \geq 1$. Consider now
	\[ \nu_n(s) = \begin{cases}
		\lambda_n(s) & \text{ if } s < t,\\
		\mu_n(s-t) + t_n & \text{ if } s \geq t.
	\end{cases} \]
	
	Then $\nu_n$ lies in $\Lambda$. Indeed, $\nu_n$ is continuous and strictly increasing in each interval $[0,t)$ and $[t,\infty)$ so it only has to be checked that it is continuous at $t$, which follows from the facts that $\nu_n$ is c\`adl\`ag by construction with
	\[ \nu_n(t-) = \lambda_n(t) = t_n = \mu_n(0) + t_n = \nu_n(t). \]
	Moreover, the triangular inequality yields
	\[ \norm{\nu_n - \Id}_\infty \leq \norm{\lambda_n - \Id}_\infty + \norm{\mu_n - \Id}_\infty + |t - t_n| \]
	and so $(\nu_n) \in \Lambda_\infty$. For any $0 \leq s \leq m$, we have by definition of $\Ccal$ and $\nu_n$
	\[ v \left( \Ccal(f_n, t_n, h_n)(\nu_n(s)), \Ccal(f, t, h)(s) \right) = \begin{cases}
		v \left( f_n(\lambda_n(s)), f(s) \right) & \text{ if } s < t,\\
		v \left( h_n(\mu_n(s-t)), h(s-t) \right) & \text{ if } s \geq t
	\end{cases} \]
	and so
	\[ v_m \left( \Ccal(f_n, t_n, h_n) \circ \nu_n, \Ccal(f, t, h) \right) \leq v_m \left(f_n \circ \lambda_n, f \right) \vee v_m \left( h_n \circ \mu_n, h \right) \]
	which gives $\Ccal(f_n, t_n, h_n) \to \Ccal(f, t, h)$ by considering $m \in M$.
\end{proof}

\begin{lemma} [Limit of the truncation map] \label{lemma:limit-truncation}
	Let $S_n \in \Scal$, $S \in \Scal_+$ and $f_n, \zeta_k \in D$. If $S_n \to S$ and $E^{S_n}(f_n) \to \zeta = (\zeta_k, k \geq 1)$, then $\Phi^{S_n}(f_n) \to \Psi^{S}(\zeta)$.
\end{lemma}

\begin{proof}
	We treat the case where $|S| = +\infty$, the other case can be treated similarly. Write $S = (s_k)$ and $S_n = (s_{n,k})$ and define $s'_k = (s_{2k} + s_{2k+1}) / 2$ for $k \geq 0$ (which is finite by the assumption $|S| = +\infty$ on $S$). For $K \geq 0$ let
	\[ \phi_{n,K}(t) = \Phi^{S_n}(f_n)(t \wedge s'_K) \ \text{ and } \ \psi_K(t) = \Psi^{S}(\zeta)(t \wedge s'_K), \ t \geq 0. \]
	
	Since $s'_K \to +\infty$ as $K \to +\infty$, the result $\Phi^{S_n}(f_n) \to \Psi^S(\zeta)$ will be proved if we can prove that $\phi_{n,K} \to \psi_K$ for every $K \geq 0$. We prove this by induction on $K \geq 0$. The result for $K = 0$ is trivial, since $\psi_{0} = \zero$ and $\phi_{n,0} = \zero$ for $n$ large enough, so assume that $K \geq 1$. First, note that by definition we have $\Psi^S(f)(t) = a$ for $s_{2K} \leq t < s_{2K+1}$ and so $\psi_K$ is continuous at $s'_K$ (here we use that $s_{2K} < s_{2K+1}$ because $S$ is assumed to be a strict subdivision). Consider $n$ large enough such that $s_{n,2K} < s'_K < s_{n,2K+1}$, so that
	\[ \phi_{n,K+1} = \Ccal \left( \phi_{n,K}, s'_K, \Ccal \left( \zero, s_{n,2K+1} - s'_K, e^{S_n}_{K+1}(f_n) \right) \right). \]
	
	Since $\zero$ is continuous we have $\Ccal(\zero, s_{n,2K+1} - s'_K, e^{S_n}_{K}(f_n)) \to \Ccal(\zero, s_{2K+1} - s'_K, \zeta_K)$ by Lemma~\ref{lemma:continuity-concatenation}. By induction hypothesis we have $\phi_{n,K} \to \psi_K$ and since $\psi_K$ is continuous at $s'_K$ we obtain that $\phi_{n,K+1} \to \Ccal(\psi_K, s'_K, \Ccal(\zero, s_{2K+1} - s'_K, \zeta_{K+1}))$ again by Lemma~\ref{lemma:continuity-concatenation}. By construction this last process is equal to $\psi_{K+1}$, hence the result.
\end{proof}

\subsection{Proof of Theorem~\ref{thm:main}} \label{sub:proof}

In the rest of this subsection we assume that the assumptions of Theorem~\ref{thm:main} hold. Moreover, we assume without loss of generality (by considering a subset thereof) that the set $C$ appearing in the statement of the theorem is countable; this allows to simplify the statement of Lemma~\ref{lemma:thinning}.

Fix some $\varepsilon \in \set$ and let $N_\varepsilon \in \{0, \ldots, +\infty\}$ be the number of $\varepsilon$-big excursions of $X$. By regeneration, we may consider a sequence $((\widetilde g_{\varepsilon, k}, e_{\varepsilon, k}), k \geq 1)$ of i.i.d.\ pairs with common distribution $(g_\varepsilon, e_\varepsilon)(X)$ such that $e_{\varepsilon, k}$ for $1 \leq k \leq N_\varepsilon$ is the $k$th big excursion of $X$ with left endpoint $g_{\varepsilon, k} = \widetilde g_{\varepsilon, 1} + \cdots + \widetilde g_{\varepsilon, k}$ and right endpoint $d_{\varepsilon, k} = g_{\varepsilon, k} + T(e_{\varepsilon, k})$. Consider then the $[0,\infty]$-valued sequence $\Pi_\varepsilon = (0, g_{\varepsilon, 1}, d_{\varepsilon, 1}, \ldots)$, so that $\Pi_\varepsilon \in \Scal_+$ since $\Ncal$ has infinite mass. Define finally the sequence $E_\varepsilon = (e_{\varepsilon, k}, k \geq 1)$: then the two sequences $E_\varepsilon$ and $E^{\Pi_\varepsilon}(X)$ coincide until the first infinite excursion of $X$ and in particular $\Phi^{\Pi_\varepsilon}(X) = \Psi^{\Pi_\varepsilon}(E_\varepsilon) = \Phi_\varepsilon(X)$.

Since $X_n$ is also regenerative, we can do the same construction and consider $N_\varepsilon^n$ the number of big excursions of $X_n$ and a sequence $((\widetilde g^n_{\varepsilon, k}, e^n_{\varepsilon, k}), k \geq 1)$ of i.i.d.\ pairs with common distribution $(g_\varepsilon, e_\varepsilon)(X_n)$ such that $e^n_{\varepsilon, k}$ for $1 \leq k \leq N^n_\varepsilon$ is the $k$th big excursion of $X_n$ with left endpoint $g^n_{\varepsilon, k} = \widetilde g^n_{\varepsilon, 1} + \cdots + \widetilde g^n_{\varepsilon, k}$ and right endpoint $d^n_{\varepsilon, k} = g^n_{\varepsilon, k} + T(e^n_{\varepsilon, k})$. Consider then the $[0,\infty]$-valued sequence $\Pi^n_\varepsilon = (0, g^n_{\varepsilon, 1}, d^n_{\varepsilon, 1}, \ldots) \in \Scal$ (since $X_n$ is not assumed to have infinite mass $\Pi^n_\varepsilon$ may fail to be a strict subdivision by having two big excursions following one another) and define $E^n_\varepsilon = (e^n_{\varepsilon, k}, k \geq 1)$, so that $\Phi^{\Pi^n_\varepsilon}(X_n) = \Psi^{\Pi^n_\varepsilon}(E^n_\varepsilon) = \Phi_\varepsilon(X_n)$.

\begin{lemma}[Continuity of thinning] \label{lemma:thinning}
	We have $\big( (\Pi^n_\varepsilon, E^n_\varepsilon), \varepsilon \in \set \big) \Rightarrow \big( (\Pi_\varepsilon, E_\varepsilon), \varepsilon \in \set \big)$.
\end{lemma}

\begin{proof}
	Let $\set'$ be any finite subset of $\set$: since $\set$ has been assumed to be countable, to prove the result it is enough to show that $((\Pi^n_\varepsilon, E^n_\varepsilon), \varepsilon \in \set') \Rightarrow ((\Pi_\varepsilon, E_\varepsilon), \varepsilon \in \set')$. Consider the case where $\set' = \{ \varepsilon_0 < \varepsilon_1 \} \subset \set$, the general case following similarly by induction. For $i = 0,1$ and $k \geq 1$, let $q_{i, k} = (g_{\varepsilon_i, k}, e_{\varepsilon_i, k}, d_{\varepsilon_i, k}, \varphi(e_{\varepsilon_i, k}))$, $q^n_{i, k} = (g^n_{\varepsilon_i, k}, e^n_{\varepsilon_i, k}, d^n_{\varepsilon_i, k}, \varphi(e^n_{\varepsilon_i, k}))$, $Q^n_i = (q^n_{i, k}, k \geq 1)$ and $Q_i = (q_{i, k}, k \geq 1)$: we will show that $(Q^n_0, Q^n_1) \Rightarrow (Q_0, Q_1)$, which clearly implies the desired result. It follows readily from~\eqref{eq:cond} that $Q_i^n \Rightarrow Q_i$ for $i = 0, 1$, so we only have to prove that these two convergences hold jointly.

	Since $\varepsilon_0 < \varepsilon_1$, an excursion which is $\varepsilon_1$-big is also $\varepsilon_0$-big and so $Q_1^n$ is a subsequence of $Q_1^n$, and $Q_1$ a subsequence of $Q_0$. More explicitly, if $u^n_k$, resp.\ $u_k$, is the index of the $k$th element of the sequence $E^n_{\varepsilon_0}$, resp.\ $E_{\varepsilon_0}$, which is $\varepsilon_1$-big, then we have $q^n_{1, k} = q^n_{0, u^n_k}$ and $q_{1, k} = q_{0,u_k}$.
	
	In particular, $U_n = (u^n_k, k \geq 1)$ is a renewal process with step distribution the geometric random variable with parameter $\P(\varphi(e_{\varepsilon_0}(X_n)) > \varepsilon_1)$. Similarly, $U = (u_k, k \geq 1)$ is a renewal process with geometric step distribution with parameter $\P(\varphi(e_{\varepsilon_0}(X)) > \varepsilon_1)$. Since by assumption, $\varphi(e_{\varepsilon_0}(X_n)) \Rightarrow \varphi(e_{\varepsilon_0}(X))$ and
	\[ \P\left( \varphi(e_{\varepsilon_0}(X)) = \varepsilon_1 \right) = \Ncal(\varphi = \varepsilon_1 \, | \, \varphi > \varepsilon_0) = 0 \]
	because $\varepsilon_1 \in \set$, we get $\P(\varphi(e_{\varepsilon_0}(X_n)) > \varepsilon_1) \to \P(\varphi(e_{\varepsilon_0}(X)) > \varepsilon_1)$ and so $U_n \Rightarrow U$. We now show that the joint convergence $(Q^n_0, U_n) \Rightarrow (Q_0, U)$ holds, which will conclude the proof since $Q^n_1 = f(Q^n_0, U_n)$ and $Q_1 = f(Q_0, U)$ for some deterministic and continuous map $f$.
	
	We show that $(Q^n_0, u^n_1) \Rightarrow (Q_0, u_1)$, the general case following similarly since $U_n$ and $U$ are renewal processes. Since $Q^n_0 \Rightarrow Q_0$ and $u^n_1 \Rightarrow u_1$ the sequence $(Q^n_0, u^n_1)$ is tight. Let $(Q', u')$ be any accumulation point and assume without loss of generality, using Skorohod's representation theorem, that $(Q^n_0, u^n_1) \to (Q', u')$. In particular, $Q'$ is equal in distribution to $Q_0$ and $u'$ to $u_1$. Since $u^n_1 \to u'$ and since these are integer-valued we have $u^n_1 = u'$ for all $n$ large enough. Let $Q' = (q'_k, k \geq 1)$ with $q'_k = (\gamma_k, \zeta_k, \delta_k, \phi_k)$, in particular it holds that $\varphi(e^n_{\varepsilon, k}) \to \phi_k$. Consider $n$ such that $u^n_1 = u'$: then $\varphi(e^n_{\varepsilon, k}) \leq \varepsilon_1$ for $k < u'$ and $\varphi(e^n_{\varepsilon, u'}) > \varepsilon_1$ which implies, letting $n \to +\infty$, that $\phi_k \leq \varepsilon_1$ for $k < u'$ while $\phi_{u'} \geq \varepsilon_1$. But $\phi_{u'} = \varepsilon_1$ does almost surely not occur since by choice $\Ncal(\varphi=\varepsilon_1\,)=0$, so that $\phi_{u'} > \varepsilon_1$ and $u' = \inf\{ k \geq 1: \phi_k > \varepsilon_1 \}$. This proves that $(Q', u')$ is equal in distribution to $(Q_0, u_1)$ and ends the proof.
\end{proof}

\begin{rem}
	Lemma~\ref{lemma:thinning} is the only place in the proof of Theorem~\ref{thm:main} where the assumption $\varphi(e_\varepsilon(X_n)) \Rightarrow \varphi(e_\varepsilon(X))$ is needed.
\end{rem}

We now conclude the proof of Theorem~\ref{thm:main}. Let $Y$ be any accumulation point of the tight sequence $(X_n)$, and assume without loss of generality that $X_n \Rightarrow Y$. Denote $\Pi^n = (\Pi^n_\varepsilon, \varepsilon \in \set)$, $E^n = (E^n_\varepsilon, \varepsilon \in \set)$, $\Pi = (\Pi_\varepsilon, \varepsilon \in \set)$ and $E = (E_\varepsilon, \varepsilon \in \set)$. The previous lemma shows that $(\Pi^n, E^n) \Rightarrow (\Pi, E)$ and so the sequence $(X_n, \Pi^n, E^n)$ is tight. Let $(Y', \Pi', E')$ be any accumulation point, so that $Y'$ is equal in distribution to $Y$ and $(\Pi', E')$ to $(\Pi, E)$, and assume without loss of generality using Skorohod's embedding theorem that the almost sure convergence $(X_n, \Pi^n, E^n) \to (Y', \Pi', E')$ holds.

Writing $\Pi' = (\Pi'_\varepsilon, \varepsilon \in \set)$ and $E' = (E'_\varepsilon, \varepsilon \in \set)$, Lemma~\ref{lemma:limit-truncation} implies (since $\Pi_\varepsilon' \in \Scal_+$) that $\Phi^{\Pi^n_\varepsilon}(X_n) \to \Psi^{\Pi_\varepsilon'}(E_\varepsilon')$. Hence the sequence $(\Phi^{\Pi^n_\varepsilon}(X_n))$ is relatively compact and so Lemma~\ref{lemma:continuity-truncation} implies that $\Phi^{\Pi^n_\varepsilon}(X_n) \to \Phi^{\Pi_\varepsilon'}(Y')$. In particular, $\Phi^{\Pi_\varepsilon'}(Y') = \Psi^{\Pi_\varepsilon'}(E_\varepsilon')$ for every $\varepsilon \in \set$: since $0 \in \overline \set$ we now want to let $\varepsilon \to 0$ (while in $\set$).

Since by construction, $\Psi^{\Pi_\varepsilon}(E_\varepsilon) = \Phi_\varepsilon(X)$ (cf.\ the beginning of the proof) and $(\Pi', E')$ is equal in distribution to $(\Pi_\varepsilon, E_\varepsilon, \varepsilon \in \set)$, we see that the family $(\Psi^{\Pi_\varepsilon'}(E_\varepsilon'), \varepsilon \in \set)$ is equal in distribution to $(\Phi_\varepsilon(X), \varepsilon \in \set)$. Lemma~\ref{lemma:Phi} shows that $\Phi_\varepsilon(X) \to X$ as $\varepsilon \to 0$, which ensures the existence of a c\`adl\`ag process $X'$, defined on the same probability space as $Y'$, $\Pi'$ and $E'$ and such that $X'$ is equal in distribution $X$ and $\Psi^{\Pi'_\varepsilon}(E'_\varepsilon) \to X'$ as $\varepsilon \to 0$. Then, by construction and Lemma~\ref{lemma:relation-operators}, we see that $\Psi^{\Pi'_\varepsilon}(E'_\varepsilon) = \Phi_\varepsilon(X')$, and in particular $\Phi^{\Pi'_\varepsilon}(Y') = \Phi_\varepsilon(X')$.
	
We now prove that $Y'=X'$, which will prove that $Y$ and $X$ are equal in distribution and will conclude the proof. Consider $t \geq 0$ in an open excursion interval of $X'$. By definition of $\Phi_\varepsilon$, for $\varepsilon\in \set$ smaller than the size of the excursion of $X'$ straddling $t$, we have $\Phi_\varepsilon(X')(t)=X'(t)\not=a$. Since $\Phi^{\Pi'_\varepsilon}(Y') = \Phi_\varepsilon(X')$ we obtain that $\Phi^{\Pi'_\varepsilon}(Y')(t)= Y'(t)=X'(t)$ for $\varepsilon$ small enough (by definition, if $\Phi^S(f)(t) \neq a$ then $\Phi^{S}(f)(t) = f(t)$). Since $\Ncal$ is infinite and $X'$ is equal in distribution to $X$, the zero set of $X'$ has empty interior, so that the union $R$ of all open excursion intervals is dense in $[0,\infty)$. We have just proved that $Y'$ and $X'$ coincide on $R$, since they are c\`adl\`ag they must coincide everywhere. The proof of Theorem~\ref{thm:main} is complete.

\section{Checking assumptions of Theorem \ref{thm:main}} \label{sec:extensions}

\subsection{Joint convergence} 

\label{sub:joint-conv}

Theorem~\ref{thm:main} requires the joint convergence~\eqref{eq:cond} of the four sequences $g_\varepsilon(X_n)$, $e_\varepsilon(X_n)$, $T(e_\varepsilon(X_n))$ and $\varphi(e_\varepsilon(X_n))$. Thanks to forthcoming Lemma \ref{lemma:joint-conv}, this joint convergence follows automatically from the convergence of the three individual sequences $g_\varepsilon(X_n)$, $(e_\varepsilon, \varphi \circ e_\varepsilon)(X_n)$ and of $T(e_\varepsilon(X_n))$. Sufficient conditions for the convergence of $g_\varepsilon(X_n)$ are provided in forthcoming Lemma~\ref{lem:conv g} and we informally discuss after Lemma~\ref{lemma:joint-conv} the convergence of $\varphi \circ e_\varepsilon$.

\begin{lemma} \label{lemma:liminf}
	If $e_n \to e$ with $e_n \in \Ecal$ and $e \in D$, then $T(e) \leq \liminf_n T(e_n)$.
\end{lemma}

\begin{proof}
	The result holds if $\liminf_n T(e_n) = +\infty$, so assume that $\liminf_n T(e_n) < +\infty$. Let $t > \liminf_n T(e_n)$ be a continuity point of $e$ and $u_n$ be a subsequence such that $T(e_{u_n}) \to \liminf_n T(e_n)$. Then $t > T(e_{u_n})$ for $n$ large enough and for those $n$ we have $e_{u_n}(t) = a$. Since in addition $e_{u_n} \to e$ and $e$ is continuous at $t$, we get $e(t) = a$ and in particular, $T(e) \leq t$. Letting $t \to \liminf_n T(e_n)$ gives the result.
\end{proof}

\begin{lemma} \label{lemma:joint-conv}
	Let $\varepsilon > 0$. If the three convergences $g_\varepsilon(X_n) \Rightarrow g_\varepsilon(X)$, $T(e_\varepsilon(X_n)) \Rightarrow T(e_\varepsilon(X))$ and $(e_\varepsilon, \varphi\circ e_\varepsilon)(X_n) \Rightarrow (e_\varepsilon, \varphi\circ e_\varepsilon)(X)$ hold, then~\eqref{eq:cond} holds.
\end{lemma}

\begin{proof}
	Assume that the three individual convergences hold. Since $X_n$ is regenerative, $g_\varepsilon(X_n)$ and $(e_\varepsilon, T\circ e_\varepsilon, \varphi\circ e_\varepsilon)(X_n)$ are independent, and the same holds for $X$, it is sufficient to prove that $(e_\varepsilon, T\circ e_\varepsilon, \varphi\circ e_\varepsilon)(X_n) \Rightarrow (e_\varepsilon, T\circ e_\varepsilon, \varphi\circ e_\varepsilon)(X)$. Since $(e_\varepsilon, \varphi\circ e_\varepsilon)(X_n) \Rightarrow (e_\varepsilon, \varphi\circ e_\varepsilon)(X)$ and $T(e_\varepsilon(X_n)) \Rightarrow T(e_\varepsilon(X))$, the joint sequence $((e_\varepsilon, T\circ e_\varepsilon, \varphi\circ e_\varepsilon)(X_n))$ is tight and we only have to identify accumulation points. So let $(e, \tau, \phi)$ be any accumulation point and assume without loss of generality that $(e_\varepsilon, T\circ e_\varepsilon, \varphi\circ e_\varepsilon)(X_n) \Rightarrow (e, \tau, \phi)$, where $\phi=\varphi(e)$: the result will be proved if we can show that $\tau = T(e)$.
	
	First, note that $\tau$ is equal in distribution to $T(e)$. Indeed, since projections are continuous $e$ is equal in distribution to $e_\varepsilon(X)$ and $\tau$ is equal in distribution to $T(e_\varepsilon(X))$. Second, note that the continuous mapping theorem together with Lemma~\ref{lemma:liminf} implies that $T(e) \leq \tau$. Hence, since $\tau$ and $T(e)$ are equal in distribution they must be equal almost surely. This proves the result.
\end{proof}

Unfortunately, there seems to be no general recipe to prove the joint convergence of $e_\varepsilon(X_n)$ and $\varphi(e_\varepsilon(X_n))$. Nonetheless, there are many natural examples where the convergence of $(e_\varepsilon, T \circ e_\varepsilon)(X_n)$ (which automatically holds when each individual sequence $e_\varepsilon(X_n)$ and $T(e_\varepsilon(X_n))$ converges) helps controlling $\varphi(e_\varepsilon(X_n))$. For instance, if $f_n \to f$ the convergence $v_\infty(f_n) \to v_\infty(f)$ needs not hold. Nonetheless, it holds if $f_n \to f$ and $T(f_n) \to T(f)$. Thus in probabilistic terms, if $\varphi = v_\infty$ then the two convergences $e_\varepsilon(X_n) \Rightarrow e_\varepsilon(X)$ and $T(e_\varepsilon(X_n)) \Rightarrow T(e_\varepsilon(X))$ imply $(e_\varepsilon, \varphi \circ e_\varepsilon)(X_n) \Rightarrow (e_\varepsilon, \varphi \circ e_\varepsilon)(X)$. Similar remarks apply for instance to maps $\varphi$ of the additive form $e \in \Ecal \mapsto \int_0^{T(e)} f(v(e(u))) du$.

\subsection{Tightness} \label{sub:tightness}

The weak convergence assumption~\eqref{eq:cond} of Theorem~\ref{thm:main} does not imply the tightness of the sequence $(X_n)$. Consider for instance $X_n$ obtained by modifying a Brownian motion $B$ where one replaces excursions of $B$ with length $\tau \in [1/n,2/n]$ by a deterministic triangle with height $n$ and basis~$\tau$. Intuitively, this example shows that when big excursions converge, one is essentially left with the problem of controlling the height of small excursions. Indeed, tightness of a sequence of c\`adl\`ag paths is essentially concerned with controlling oscillations, and for small excursions we do not lose much by upper bounding their oscillations by their height.

We now introduce some notation necessary in order to study tightness. For $f \in D$ and $m, \delta > 0$, let
\[ w_m'(f, \delta) = \inf_{(I_k)} \ \max_k \, \sup_{x,y \in I_k} v(f(x), f(y)) \]
where the infimum extends over all partitions of the interval $[0,m)$ into subintervals $I_k = [s,t)$ such that $t-s > \delta$ when $t < m$. Then the sequence $X_n$ is tight if and only if the sequence $(X_n(t))$ is tight (in $V$) for every $t$ in a dense subset of $[0,\infty)$ and for every $m$ and $\eta > 0$
\[
	\lim_{\delta \to 0} \limsup_{n \to +\infty} \P \left( w'_m(X_n, \delta) \geq \eta \right) = 0,
\]
see for instance Jacod and Shiryaev~\cite{Jacod03:0}. Tightness criteria in $V$ depend on $V$. We will discuss the most important case $V = \R^d$ for some $d \geq 1$. In this case, $(X_n(t))$ is tight for every $t$ in a dense subset of $[0,\infty)$ if and only if for every $m \geq 0$,
\[
	\lim_{b \to +\infty} \limsup_{n \to +\infty} \P \left( \norm{X_n}_m^d \geq b \right) = 0,
\]
where from now on, $\norm{f}_m^d$ for $f: t \in [0,\infty) \mapsto (f_k(t), 1 \leq k \leq d) \in \R^d$ is defined by $\norm{f}_m^d = \max_{1 \leq k \leq d} \norm{f_k}_m$.

In the sequel, we will also discuss C-tightness: a sequence of processes is C-tight if it is tight and every accumulation point is almost surely continuous. Necessary and sufficient conditions for $X_n$ to be C-tight are the same as for tightness, replacing $w'$ by the modulus of continuity $w$ defined for $f \in D$ and $m, \varepsilon > 0$ by
\[ w_m(f, \delta) = \sup \left\{ v(f(t), f(s)) : 0 \leq s, t \leq m, |t-s| \leq \delta \right\}, \]
see for instance Jacod and Shiryaev~\cite{Jacod03:0}. In the sequel we consider the truncation operator $\overline \Phi_\varepsilon$, which truncates $\varepsilon$-big excursions to $a$ (remember the definition~\eqref{eq:Phi} of $\Phi_\varepsilon$):
\[ 	\overline \Phi_\varepsilon(f)(t) = \begin{cases}
	f(t) & \text{ if } \varphi(e^{st}(f,t)) \leq \varepsilon,\\
	a & \text{ else.}
\end{cases}
\]

Note that $f(t)$ is either equal to $\Phi_\varepsilon(f)(t)$ or to $\overline \Phi_\varepsilon(f)(t)$.

\begin{lemma} \label{lemma:tightness-weak}
	Consider some set $\set \subset (0,\infty)$ such that $0 \in \overline \set$ and assume that for every $\varepsilon \in \set$, the three sequences $(g_\varepsilon(X_n))$, $(e_\varepsilon(X_n))$ and $(T(e_\varepsilon(X_n)))$ are tight and every accumulation point $(\gamma_\varepsilon, \tau_\varepsilon) \in [0,\infty]^2$ of the sequence $((g_\varepsilon, T \circ e_\varepsilon)(X_n))$ satisfies $\P(\gamma_\varepsilon = 0) = \P(\tau_\varepsilon = 0) = 0$. Then for any $m$ and $\eta > 0$,
	\begin{equation} \label{eq:bound-tightness-1}
		\lim_{\delta \to 0} \limsup_{n \to +\infty} \P \left( w'_m(X_n, \delta) \geq 4\eta \right) \leq \lim_{\varepsilon \to 0} \limsup_{n \to +\infty} \P \left( v_m(\overline \Phi_\varepsilon(X_n)) \geq \eta \right).
	\end{equation}

	If in addition, for each $\varepsilon \in \set$ the sequence $(e_\varepsilon(X_n))$ is C-tight and any of its accumulation point $\zeta_\varepsilon$ satisfies $\P(\zeta_\varepsilon(0) = a) = 1$, then $w'$ in the left hand side of~\eqref{eq:bound-tightness-1} can be replaced by $w$.
	
	Finally, in the particular case $V = \R^d$, $a = 0$ and $\varphi = \norm{\cdot}_\infty^d$, the above assumptions imply that $(X_n)$ is tight, and even C-tight if $(e_\varepsilon(X_n))$ is C-tight and all its accumulation points almost surely start at $a$.
\end{lemma}

\begin{proof}
	Let $\varepsilon \in \set$: first we prove that $(\Phi_\varepsilon(X_n))$ is tight, and even C-tight if $(e_\varepsilon(X_n))$ is C-tight with $\P(\zeta_\varepsilon(0) = a) = 1$ for every accumulation point $\zeta_\varepsilon$. Let $(u_n)$ be some subsequence, we must find a subsequence $(z_n)$ of $(u_n)$ such that $(\Phi_{\varepsilon}(X_{z_n}))$ converges weakly. By assumption, the sequence $((g_\varepsilon, e_\varepsilon, T \circ e_\varepsilon)(X_n), n \geq 1)$ is tight, so there exists $(z_n)$ a subsequence of $(u_n)$ such that $(g_\varepsilon, e_\varepsilon, T \circ e_\varepsilon)(X_{z_n}) \Rightarrow (\gamma_\varepsilon, \zeta_\varepsilon, \tau_\varepsilon)$ for some random variable $(\gamma_\varepsilon, \zeta_\varepsilon, \tau_\varepsilon) \in [0,\infty] \times D \times [0,\infty]$ with $\P(\gamma_\varepsilon = 0) = \P(\tau_\varepsilon = 0) = 0$. Let $((\gamma_{\varepsilon,k}, \zeta_{\varepsilon, k}, \delta_{\varepsilon, k}), k \geq 1)$ be an i.i.d.\ sequence with common distribution $(\gamma_\varepsilon, \zeta_\varepsilon, \gamma_\varepsilon + \tau_\varepsilon)$, $\zeta = (\zeta_{\varepsilon, k}, k \geq 1)$ and $S = (s_k)$ be the subdivision defined recursively by $s_{2k+1} = s_{2k} + \gamma_{\varepsilon, k+1}$ and $s_{2k+2} = s_{2k+1} + \tau_{\varepsilon, k+1}$. Since $\P(\gamma_\varepsilon = 0) = \P(\tau_\varepsilon = 0) = 0$, $S$ is almost surely a strict subdivision, and hence, proceeding similarly as in the proof of Theorem~\ref{thm:main}, it can be proved that $\Phi_\varepsilon(X_{z_n}) \Rightarrow \Psi^{S}(\zeta)$. This proves tightness.
	
	Assume now in addition that $\zeta_\varepsilon$ is almost surely continuous with $\zeta_\varepsilon(0) = a$: we prove that $(\Phi_\varepsilon(X_n))$ is actually C-tight by proving that $\Psi^S(\zeta)$ is continuous. Neglecting a set of zero measure, we can assume that every $\zeta_{\varepsilon, k}$ is continuous and starts at $a$, in which case it is plain from its definition that $\Psi^{S}(\zeta)$ is continuous if $\zeta_{\varepsilon, k}(\tau_{\varepsilon, k}) = a$ for each $k \geq 1$. For this it is enough to show that $\P(\zeta_\varepsilon(\tau_\varepsilon) = a) = 1$. Assume almost sure convergence $(e_\varepsilon, T \circ e_\varepsilon)(X_{z_n}) \to (\zeta_\varepsilon, \tau_\varepsilon)$, so that Lemma~\ref{lemma:liminf} implies that $T(\zeta_\varepsilon) \leq \tau_\varepsilon$. If equality holds, then we can use $e_\varepsilon(X_{z_n})(T(e_\varepsilon(X_{z_n}))) \to \zeta_\varepsilon(\tau_\varepsilon)$ (which holds since $\zeta_\varepsilon$ is continuous) to deduce that $\zeta_\varepsilon(\tau_\varepsilon) = a$. If strict inequality holds then we can use $e_\varepsilon(X_{z_n})(\tau_\varepsilon) \to \zeta_\varepsilon(\tau_\varepsilon)$ to deduce that $\zeta_\varepsilon(\tau_\varepsilon) = a$. In either case we have $\zeta_\varepsilon(\tau_\varepsilon) = a$, hence the result.
	\\
	
	We now prove~\eqref{eq:bound-tightness-1}. Let $f \in D$. For any $x,y \geq 0$, we have
	\[ v \left( f(x), f(y) \right) \leq v \left( \Phi_\varepsilon(f)(x), \Phi_\varepsilon(f)(y) \right) + v \left( f(x), \Phi_\varepsilon(f)(x) \right) + v \left( f(y), \Phi_\varepsilon(f)(y) \right). \]
	
	Moreover, it is plain from the definitions of $\Phi_\varepsilon$ and $\overline \Phi_\varepsilon$ that $v_m(\Phi_\varepsilon(f), f) \leq v_m(\overline \Phi_\varepsilon(f))$ and so
	\[ v \left( f(x), f(y) \right) \leq v \left( \Phi_\varepsilon(f)(x), \Phi_\varepsilon(f)(y) \right) + 2 v_m(\overline \Phi_\varepsilon(f)) \]
	for all $0 \leq x,y \leq m$, from which it readily follows that 
	\[ w_m'(f, \delta) \leq w_m'(\Phi_\varepsilon(f), \delta) + 2 v_m(\overline \Phi_\varepsilon(f)). \]
	
	This gives
	\[ \P \left( w_m'(X_n, \delta) \geq 4\eta \right) \leq \P \left( w_m'(\Phi_\varepsilon(X_n), \delta) \geq 2\eta \right) + \P \left( v_m(\overline \Phi_\varepsilon(X_n)) \geq \eta \right). \]
	Since $(\Phi_\varepsilon(X_n))$ is tight, letting first $n \to +\infty$, then $\delta \to 0$ and finally $\varepsilon \to 0$ gives~\eqref{eq:bound-tightness-1}. When $(\Phi_\varepsilon(X_n))$ is C-tight one can derive similar inequalities with $w$ instead of $w'$.
	
	Consider now the case $V = \R^d$, $a = 0$ and $\varphi = v_\infty = \norm{\cdot}_\infty^d$. Then by definition we obtain $v_m(\overline \Phi_\varepsilon(f)) \leq \varepsilon$ since $\overline \Phi_\varepsilon$ truncates all excursions with height larger than $\varepsilon$ and in particular, $\P( \norm{\overline \Phi_\varepsilon(X_n)}_m^d \geq \eta) = 0$ for every $\varepsilon < \eta$. Hence the right hand side of~\eqref{eq:bound-tightness-1} is equal to $0$ and to show that $(X_n)$ is tight it remains to control its supremum. By definition we have $\P(\norm{X_n}_m^d \geq b) = \P(\norm{\Phi_{1}(X_n)}_m^d \geq b)$  for $b \geq 1$ and since $(\Phi_1(X_n))$ is tight we obtain
	\[
		\lim_{b \to +\infty} \limsup_{n \to +\infty} \P \left( \norm{X_n}_m^d \geq b \right) = 0
	\]
	which proves the tightness of $(X_n)$. C-tightness results follow similarly.
\end{proof}

Combining Lemmas~\ref{lemma:joint-conv} and~\ref{lemma:tightness-weak} we obtain the following simple version of Theorem~\ref{thm:main}.

\begin{theorem} \label{thm:main-2}
	Consider the case where $V = \R^d$, $a = 0$ and $\varphi = \norm{\cdot}_\infty^d$ and let $\set \subset (0,\infty)$ such that $0 \in \overline \set$ and $\Ncal(\varphi = \varepsilon)=0$  for all $\varepsilon\in \set$. If $g_\varepsilon(X_n) \Rightarrow g_\varepsilon(X)$, $e_\varepsilon(X_n) \Rightarrow e_\varepsilon(X)$ and $T(e_\varepsilon(X_n)) \Rightarrow T(e_\varepsilon(X))$ for every $\varepsilon \in \set$, then $X_n \Rightarrow X$.
\end{theorem}

\begin{rem}
	This result suggests that the choice of the map $\varphi = \norm{\cdot}_\infty^d$ is particularly interesting, because it takes care of tightness. In Lemma~\ref{lemma:varphi'} we discuss conditions under which a control of excursions measured according to $\varphi$ can give a control on excursions measured according to $\norm{\cdot}_\infty^d$.
\end{rem}

\subsection{Conditions based on excursion measures} \label{sub:excursion-measures}

In this subsection we discuss various conditions on the excursion measures of the processes $X_n$ that guarantee that parts of the assumptions of Theorem~\ref{thm:main} hold. In the rest of this subsection, we make the following assumption, which we believe to be the most relevant one in practice.

\begin{assumption}
	In the rest of this subsection, we assume that $X_n$ has only finitely many excursions in any bounded interval, i.e., its excursion measure has finite mass.
\end{assumption}

In the rest of this subsection, we denote by $\Ncal_n$ the probability measure of excursions of $X_n$ away from $a$, i.e., the law of the process $X_n$ shifted at its first exit time of $\{a\}$ and stopped upon returning to $a$. We denote by $b_n$ the parameter of the exponentially distributed holding times at $a$. We also assume we are given a sequence $(c_n)$ of positive real numbers such that $c_n\to +\infty$ and (within this section) we define the set $\set$ as the complement of the set of atoms of $\Ncal$:
\[ \set = \left\{ \varepsilon > 0: \Ncal(\varphi = \varepsilon) = 0 \right\}. \]

In the sequel we adopt the canonical notation when dealing with excursions. We will say that~\eqref{h:varphi} is satisfied if for any $\varepsilon \in \set$,
\begin{equation} \label{h:varphi} \tag{\bf{H1}}
	\lim_{n\to+\infty}c_n\Ncal_n(\varphi >\varepsilon) = \Ncal (\varphi >\varepsilon).
\end{equation}

We will say that~\eqref{h:T+varphi} is satisfied if for any $\lambda>0$ and $\varepsilon \in \set$,
\begin{equation} \label{h:T+varphi} \tag{\bf{H2}}
	\lim_{n\to+\infty}c_n\Ncal_n\left(1-e^{-\lambda T};\varphi \le\varepsilon\right) = \Ncal\left(1-e^{-\lambda T};\varphi \le\varepsilon\right).
\end{equation}

We will say that~\eqref{h:T} is satisfied if for any $\lambda>0$
\begin{equation} \label{h:T} \tag{\bf{H3}}
	\lim_{n\to+\infty}c_n\Ncal_n\left(1-e^{- \lambda T}\right) = \Ncal\left(1-e^{-\lambda T}\right).
\end{equation}

It is easy to see (and will be established in the proof of Theorem~\ref{thm:ex}) that~\eqref{h:varphi} and~\eqref{h:T+varphi} together imply~\eqref{h:T} when we have in addition $\Ncal(\varphi = +\infty) = 0$. Moreover,~\eqref{h:varphi},~\eqref{h:T+varphi} and~\eqref{h:T} together give the convergence, for $\varepsilon \in \set$, of the two sequences $(\varphi(e_\varepsilon(X_n)))$ and $(T(e_\varepsilon(X_n)))$. Indeed, since
\[ \P \left( \varphi(e_\varepsilon(X_n)) > \delta \right) = \frac{\Ncal_n \left( \varphi > \delta \vee \varepsilon \right)}{\Ncal_n \left( \varphi > \varepsilon \right)} \ \text{ and } \ \E \left( 1 - e^{-\lambda T(e_\varepsilon(X_n))} \right) = \frac{\Ncal_n \left( 1 - e^{-\lambda T} ; \varphi > \varepsilon \right)}{\Ncal_n(\varphi > \varepsilon)} \]
we have the following elementary result.

\begin{lemma} \label{lemma:H}
	If~\eqref{h:varphi} holds, then $\varphi(e_\varepsilon(X_n)) \Rightarrow \varphi(e_\varepsilon(X))$ for every $\varepsilon \in \set$. If~\eqref{h:varphi},~\eqref{h:T+varphi} and~\eqref{h:T} hold, then $T(e_\varepsilon(X_n)) \Rightarrow T(e_\varepsilon(X))$ for every $\varepsilon \in \set$.
\end{lemma}

These assumptions also imply the convergence of the sequence $(g_\varepsilon(X_n))$ under an additional assumption on $b_n$ and $c_n$. Note that the sequence $(g_\varepsilon(X_n))$ plays a particular role in~\eqref{eq:cond}, since by regeneration $g_\varepsilon(X_n)$ is independent of $e_\varepsilon(X_n)$, $T(e_\varepsilon(X_n))$ and $\varphi(e_\varepsilon(X_n))$ which are all direct functionals of $e_\varepsilon(X_n)$. In particular, one has to prove the convergence of $(g_\varepsilon(X_n))$ separately to check~\eqref{eq:cond}.

We stress that the excursion measure $\Ncal$ is uniquely defined only after the local time has been normalized. A local time process $(L(t), t \ge 0)$, is a nondecreasing process with values in $\R_+$, which satisfies $L(0)=0$ and for any $t>s$, $L(t)>L(s)$ if and only if there is $u\in(s,t)$ such that $X(u)=a$. It is unique up to a multiplicative constant. It is known (see Kallenberg \cite{Kallenberg02:0}) that the inverse of $L$ is a subordinator, i.e., an increasing L\'evy process. We let $d$ be the drift coefficient of this subordinator. In particular, the zero set of $X$ has positive Lebesgue measure if and only if $d>0$. 

\begin{lemma}
\label{lem:conv g}
If~\eqref{h:varphi} and~\eqref{h:T+varphi} hold, and $c_n/b_n\to d$ with $d$ as above, then $g_\varepsilon(X_n)\Rightarrow g_\varepsilon(X)$ for all $\varepsilon \in \set$.
\end{lemma}

\begin{proof}
Since the excursion measure of $X_n$ is finite, $g_\varepsilon(X_n)$ can be written as
$$
g_\varepsilon(X_n) = H_{n,0} + \sum_{i=1}^{N_{n,\varepsilon}}(T_{n,\varepsilon,i}+H_{n,i}),
$$
where all random variables are independent, and the sum is understood to be zero when $N_{n, \varepsilon} = 0$. The $(T_{n,\varepsilon,i}, i \geq 1)$ are i.i.d.\ distributed as the lifetime $T$ under $\Ncal_n(\,\cdot \mid \varphi \le\varepsilon)$, the $(H_{n,i}, i \geq 0)$ are independent exponential random variables with common parameter $b_n$ and
 $\P(N_{n,\varepsilon}=k)=\Ncal_n(\varphi>\varepsilon)\big(\Ncal_n(\varphi\le \varepsilon)\big)^k$.
As a consequence,
$$
\E_n \big(e^{-\lambda g_\varepsilon(X_n)}\big) = \frac{\Ncal_n(\varphi>\varepsilon) \E(e^{-\lambda H_{n,1}})}{1-\Ncal_n\big(e^{-\lambda T}; \varphi\le\varepsilon\big)\, \E\big( e^{-\lambda H_{n,1}}\big)} = \frac{b_n / (\lambda + b_n)}{1 + A_{n,\varepsilon} - R_{n, \varepsilon}}
$$
where
$$
A_{n,\varepsilon}:= \frac{\lambda/b_n}{(\lambda/b_n + 1)\Ncal_n(\varphi>\varepsilon)}+ \frac{\Ncal_n\big(1-e^{-\lambda T};\varphi\le\varepsilon\big)}{\Ncal_n(\varphi>\varepsilon)}
$$
and
$$
R_{n,\varepsilon} := \frac{\lambda}{\lambda + b_n} \,\left(1+\frac{\Ncal_n\big(1-e^{-\lambda T};\varphi\le\varepsilon\big)}{\Ncal_n(\varphi>\varepsilon)}\right).
$$

Under the assumptions of the lemma, we have $b_n \to +\infty$, in particular $A_{n,\varepsilon} \to A_\varepsilon$ and $R_{n,\varepsilon} \to 0$, where
$$
A_{\varepsilon}:=\frac{d\lambda}{\Ncal(\varphi>\varepsilon)}+ \frac{\Ncal\big(1-e^{-\lambda T};\varphi\le\varepsilon\big)}{\Ncal(\varphi>\varepsilon)}.
$$

On the other hand, basic excursion theory ensures that $\E\big( e^{-\lambda g_\varepsilon(X)}\big) =(1+A_{\varepsilon})^{-1}$, so we have proved the convergence of the Laplace transform of $g_\varepsilon(X_n)$ to the Laplace transform of $g_\varepsilon(X)$ which proves the result.
\end{proof}

The previous two lemmas show how one can exploit the assumptions~\eqref{h:varphi},~\eqref{h:T+varphi} and~\eqref{h:T} to show that~\eqref{eq:cond} holds. We now investigate in view of Lemma~\ref{lemma:tightness-weak} how these assumptions can be used to show tightness. In the following lemma, $\Ncal_n(\, \cdot \, | \, T = +\infty)$ refers to the null measure when $\Ncal_n(T = +\infty) = 0$.

\begin{lemma} \label{lemma:overline-phi}
	For any $n \geq 1$, $m \geq 0$ and $\varepsilon, \eta, \lambda$ and $\alpha > 0$, it holds that
	\begin{multline} \label{eq:bound}
		\P \left( v_m(\overline \Phi_\varepsilon(X_n)) \geq \eta \right) \leq \alpha c_n \Ncal_n \left( v_\infty \geq \eta \, | \, \varphi \leq \varepsilon, T < +\infty \right)\\
		+ e^{\lambda m} \exp \left( - \frac{ \lfloor \alpha c_n \rfloor \lambda}{\lambda + b_n} - \lfloor \alpha c_n \rfloor \Ncal_n \left( 1 - e^{-\lambda T} \, | \, T < +\infty \right) \right) + \Ncal_n \left( v_m \geq \eta, \varphi \leq \varepsilon \, | \, T = +\infty \right).
	\end{multline}
\end{lemma}

\begin{proof}
	Let $\alpha_n = \lfloor \alpha c_n\rfloor$, and consider first the case where $X_n$ has no infinite excursion, i.e., $\Ncal_n(T = +\infty) = 0$. Let $R_{\alpha,n} < +\infty$ be the right endpoint of the $\alpha_n$th excursion of $X_n$: then by monotonicity of $v_m$ in $m$, we obtain
	\[ \P \left( v_m(\overline \Phi_\varepsilon(X_n)) \geq \eta \right) \leq \P\left( R_{\alpha,n} \leq m \right) + \P \left( v_{R_{\alpha,n}} (\overline \Phi_\varepsilon(X_n)) \geq \eta \right). \]
	
	Let $N_{\alpha,n,\varepsilon}$ be the number of $\varepsilon$-big excursions $e$ of $X_n$ among the first $\alpha_n$:
	\[ \P \left( v_{R_{\alpha,n}} (\overline \Phi_\varepsilon(X_n)) \ge \eta \right) = 1-\E \left[ \left\{ \Ncal_n \left( v_\infty < \eta \, | \, \varphi \leq \varepsilon  \right) \right\}^{N_{\alpha,n,\varepsilon}} \right] \leq \alpha c_n \Ncal_n \left( v_\infty \geq \eta \, | \, \varphi \leq \varepsilon \right) \]
	using to derive the second inequality the inequalities $1 - x^n \leq n (1-x)$ and $N_{\alpha,n,\varepsilon} \leq \alpha c_n$. By definition, $R_{\alpha,n}$ is equal to
	\[ R_{\alpha,n} = \sum_{i=1}^{\alpha_n} (E_{n,i} + T_{n,i}) \]
	where all the random variables appearing in the right hand side are independent, $E_{n,i}$ is exponentially distributed with parameter $b_n$ and $T_{n,i}$ is equal in distribution to $T$ under $\Ncal_n$. Thus
	\[ \P\left( R_{\alpha,n} \leq m \right) \leq e^{\lambda m} \E \left( e^{-\lambda R_{\alpha,n}} \right) = e^{\lambda m} \left\{ \E \left( e^{-\lambda E_n} \right) \right\}^{\alpha_n} \left\{ \Ncal_n \left( e^{-\lambda T} \right) \right\}^{\alpha_n} \]
	and by convexity, we obtain that
	\[ \P\left( R_{\alpha,n} \leq m \right) \leq \exp \left( \lambda m - \alpha_n \E \left( 1 - e^{-\lambda E_n} \right) - \alpha_n \Ncal_n \left( 1 - e^{-\lambda T} \right) \right). \]
	
	This proves the result when $\Ncal_n(T = +\infty) = 0$. Assuming $\Ncal_n(T = +\infty) > 0$, we can write isolating the infinite excursion (which has law $\Ncal_n(\, \cdot \, | \, T = +\infty)$)
	\[ \P \left( v_m(\overline \Phi_\varepsilon(X_n)) \geq \eta \right) \leq \P \left( v_m(\overline \Phi_\varepsilon(X_n')) \geq \eta \right) + \Ncal_n \left( v_m \geq \eta, \varphi \leq \varepsilon \, | \, T = +\infty \right) \]
	where $X'_n(t) = X_n(t)$ if $T(e^{st}(X_n,t)) < +\infty$ and $X'_n(t) = a$ if $T(e^{st}(X_n,t)) = +\infty$ (recall that $e^{st}(f,t) \in \Ecal$ is the excursion of $f$ straddling $t$). But we have
	\[ \P\left(v_m(\overline \Phi_\varepsilon(X_n')) \geq \eta\right) \leq \P\left(v_m(\overline \Phi_\varepsilon(X_n'')) \geq \eta\right) \]
	with $X''_n$ a regenerative process with holding times with parameter $b_n$ and normalized excursion measure $\Ncal_n(\, \cdot \, | \, T < +\infty)$, to which we can apply the previous results.
\end{proof}

A corollary to the previous observations is the following statement, which illustrates a possible way to combine the previous results. See also Theorems \ref{thm:main-2} and \ref{thm:main-3} for other extensions in the case when $V=\R^d$.

\begin{theorem} \label{thm:ex}
	Assume that the zero set of $X$ has zero Lebesgue measure, that $\Ncal(\varphi = +\infty) = 0$, that~\eqref{h:varphi} and~\eqref{h:T+varphi} hold and that $c_n / b_n \to 0$. Then $g_\varepsilon(X_n) \Rightarrow g_\varepsilon(X)$, $T(e_\varepsilon(X_n)) \Rightarrow T(e_\varepsilon(X))$ and $\varphi(e_\varepsilon(X_n)) \Rightarrow \varphi(e_\varepsilon(X))$ for every $\varepsilon \in \set$.
	
	Assume in addition that $\Ncal_n(T = +\infty) = 0$ for every $n \geq 1$, that the joint convergence $(e_\varepsilon, \varphi\circ e_\varepsilon)(X_n)\Rightarrow (e_\varepsilon, \varphi\circ e_\varepsilon)(X)$ holds for every $\varepsilon \in \set$ and that
	\begin{equation} \label{eq:ass-tightness}
		\lim_{\varepsilon \to 0} \limsup_{n \to +\infty} \left[ c_n \Ncal_n \left( v_\infty \geq \eta, \varphi \leq \varepsilon \right) \right] = 0
	\end{equation}
	for every $\eta > 0$. Then $X_n \Rightarrow X$.
\end{theorem}

\begin{proof}
	First we prove that~\eqref{h:varphi} and~\eqref{h:T+varphi} together with the assumption $\Ncal(\varphi = +\infty) = 0$ imply~\eqref{h:T}. Let $\lambda > 0$ and $\varepsilon > 0$: using $0 \leq 1 - e^{-\lambda T} \leq 1$ one gets
	\[ c_n \Ncal_n\left(1-e^{- \lambda T} ; \varphi \leq \varepsilon \right) \leq c_n\Ncal_n\left(1-e^{- \lambda T}\right) \leq c_n\Ncal_n\left(1-e^{- \lambda T} ; \varphi \leq \varepsilon \right) + c_n \Ncal_n\left(\varphi > \varepsilon \right). \]
	
	Letting first $n \to +\infty$ and then $\varepsilon \to +\infty$ (while in $C$) then implies~\eqref{h:T}. Thus~\eqref{h:varphi},~\eqref{h:T+varphi} and~\eqref{h:T} all hold and the first part of the theorem is just a combination of Lemmas~\ref{lemma:H} and~\ref{lem:conv g}. Assuming in addition that $(e_\varepsilon, \varphi\circ e_\varepsilon)(X_n)\Rightarrow (e_\varepsilon, \varphi\circ e_\varepsilon)(X)$  for every $\varepsilon \in \set$ gives~\eqref{eq:cond} by Lemma~\ref{lemma:joint-conv}. Hence in order to apply Theorem~\ref{thm:main} we only need to prove tightness of $X_n$. Combining~\eqref{eq:bound-tightness-1},~\eqref{eq:bound} and the fact that $c_n / b_n \to 0$ one readily sees that one is left with showing that
	\[ \lim_{\varepsilon \to 0} \limsup_{n \to +\infty} c_n \Ncal_n \left( v_\infty \geq \eta \, | \, \varphi \leq \varepsilon \right) = 0. \]
	
	Since $\Ncal_n \left( v_\infty \geq \eta \, | \, \varphi \leq \varepsilon \right) = \Ncal_n \left( v_\infty \geq \eta, \varphi \leq \varepsilon \right) \big / \Ncal_n \left( \varphi \leq \varepsilon \right)$ and $\Ncal_n(\varphi \leq \varepsilon) \to 1$ as a consequence of~\eqref{h:varphi}, the result follows from the assumption~\eqref{eq:ass-tightness}.
\end{proof}

Let us make a final comment related to the convergence of excursion measures, following the remark made after Theorem~\ref{thm:main-2}. Lemma~\ref{lemma:tightness-weak} shows that controlling excursions measured according to $v_\infty$ makes it possible to automatically control oscillations, and even implies tightness when $V = \R^d$. It is therefore natural to ask whether such a control can be obtained from the control of excursions measured according to $\varphi$. More generally, if $e_\varepsilon^\phi(f)$ denotes the first excursion $e$ of $f$ that satisfies $\phi(e) > \varepsilon$, it is natural to ask under which conditions a control on $e_\varepsilon^{\varphi_1}$ gives a control on $e_\varepsilon^{\varphi_2}$, given two different maps $\varphi_1, \varphi_2: \Ecal \to [0,\infty]$.

\begin{lemma} \label{lemma:varphi'}
	For $i = 1,2$, let $\varphi_i: \Ecal \to [0,\infty]$ be a measurable map such that $\varphi_i(e) = 0$ if and only if $e = \zero$. Let $\set_i = \{ \varepsilon > 0: \Ncal(\varphi_i = \varepsilon) = 0 \}$ and assume that $\Ncal(\varphi_i > \varepsilon)$ is finite for every $\varepsilon > 0$.
	
	If for every $\varepsilon_1 \in \set_1$ and $\varepsilon_2 \in \set_2$, $c_n \Ncal_n(\varphi_1 > \varepsilon_1) \to \Ncal(\varphi_1 > \varepsilon_1)$, $c_n \Ncal_n(\varphi_2 > \varepsilon_2) \to \Ncal(\varphi_2 > \varepsilon_2)$ and $(e_{\varepsilon_1}^{\varphi_1}, \varphi_2 \circ e_{\varepsilon_1}^{\varphi_1})(X_n) \Rightarrow (e_{\varepsilon_1}^{\varphi_1}, \varphi_2 \circ e_{\varepsilon_1}^{\varphi_1})(X)$, then $e_{\varepsilon_2}^{\varphi_2}(X_n) \Rightarrow e_{\varepsilon_2}^{\varphi_2}(X)$ for every $\varepsilon_2 \in \set_2$.
\end{lemma}

\begin{proof}
	Let $\varepsilon_1 \in \set_1$, $\varepsilon_2 \in \set_2$ and $f: \Ecal \to [0,\infty)$ be a bounded and continuous function: we must prove that $\Ncal_n( f \, | \, \varphi_2 > \varepsilon_2) \to \Ncal( f \, | \, \varphi_2 > \varepsilon_2)$. Let $M = \sup f$: then we have
	\begin{multline*}
		\Ncal_n \left( f \indicator{\varphi_1 > \varepsilon_1} \, | \, \varphi_2 > \varepsilon_2 \right) \leq \Ncal_n \left( f \, | \, \varphi_2 > \varepsilon_2 \right) \leq \Ncal_n \left( f \indicator{\varphi_1 > \varepsilon_1} \, | \, \varphi_2 > \varepsilon_2 \right)\\
		+ M \Ncal_n \left( \varphi_1 \leq \varepsilon_1 \, | \, \varphi_2 > \varepsilon_2 \right).
	\end{multline*}
	
	By definition,
	\[
		\Ncal_n \left( f \indicator{\varphi_1 > \varepsilon_1} \, | \, \varphi_2 > \varepsilon_2 \right) = \frac{\Ncal_n \left( \varphi_1 > \varepsilon_1 \right)}{\Ncal_n \left( \varphi_2 > \varepsilon_2 \right)} \Ncal_n \left( f \indicator{\varphi_2 > \varepsilon_2} \, | \, \varphi_1 > \varepsilon_1 \right)
	\]
	so that by assumption,
	\begin{align*}
		\lim_{n \to +\infty} \Ncal_n \left( f \indicator{\varphi_1 > \varepsilon_1} \, | \, \varphi_2 > \varepsilon_2 \right) & = \frac{\Ncal \left( \varphi_1 > \varepsilon_1 \right)}{\Ncal \left( \varphi_2 > \varepsilon_2 \right)} \Ncal \left( f \indicator{\varphi_2 > \varepsilon_2} \, | \, \varphi_1 > \varepsilon_1 \right)\\
		& = \Ncal \left( f \indicator{\varphi_1 > \varepsilon_1} \, | \, \varphi_2 > \varepsilon_2 \right).
	\end{align*}
	
	Since $\Ncal_n \left( \varphi_1 \leq \varepsilon_1 \, | \, \varphi_2 > \varepsilon_2 \right) = 1-\Ncal_n \left( \varphi_1 > \varepsilon_1 \, | \, \varphi_2 > \varepsilon_2 \right)$, the previous formula for $f = 1$ (the function with constant value $1$) gives $\Ncal_n \left( \varphi_1 \leq \varepsilon_1 \, | \, \varphi_2 > \varepsilon_2 \right) \to \Ncal \left( \varphi_1 \leq \varepsilon_1 \, | \, \varphi_2 > \varepsilon_2 \right)$. In particular, we get the following bounds:
	\begin{multline*}
		\Ncal \left( f \indicator{\varphi_1 > \varepsilon_1} \, | \, \varphi_2 > \varepsilon_2 \right) \leq \liminf_{n \to +\infty} \Ncal_n \left( f \, | \, \varphi_2 > \varepsilon_2 \right) \leq \limsup_{n \to +\infty} \Ncal_n \left( f \, | \, \varphi_2 > \varepsilon_2 \right)\\
		\leq \Ncal \left( f \indicator{\varphi_1 > \varepsilon_1} \, | \, \varphi_2 > \varepsilon_2 \right) + M \Ncal \left( \varphi_1 \leq \varepsilon_1 \, | \, \varphi_2 > \varepsilon_2 \right).
	\end{multline*}
	Since $\varepsilon_1 \in \set_1$ is arbitrary, letting $\varepsilon_1 \to 0$ gives the result by monotone convergence.
\end{proof}

\subsection{Shifting the first excursion to reach level $\varepsilon$} \label{sub:another-shift}

Consider for a moment the case $V = \R$ and $\varphi = \norm{\cdot}_\infty$. It is natural, at least in the context of Markov processes, to follow an excursion $e$ conditioned on entering $(\varepsilon, \infty)$ only after the time $T^\uparrow_\varepsilon$ when it enters $(\varepsilon, \infty)$. Indeed, in the case of strong Markov processes the conditioning only affects the shifted process $\theta_{T^\uparrow_\varepsilon} (e)$ through the value $e(T^\uparrow_\varepsilon)$. Hence following the excursion after time $T^\uparrow_\varepsilon$ makes it possible to get rid of the conditioning and should therefore be an easier task than studying the whole conditioned excursion. This approach turned out to be useful in Borst and Simatos~\cite{Borst:0}, see also the closely related remark by Aldous~\cite[Section 6.3]{Aldous97:0} in the context of random graphs.

In the rest of this subsection, because of tightness issues we restrict ourselves to the case $V = \R^d$, $a = 0$ and $\varphi = \norm{\cdot}_\infty^d$. With an abuse in notation define $|x| = \max_{1 \leq k \leq d} |x_i|$ for $x = (x_i, 1 \leq i \leq d) \in \R^d$. For $\varepsilon > 0$ and $f \in D$ define the times $\widetilde T^\uparrow(f, \varepsilon) \leq T^\uparrow(f, \varepsilon)$ as follows:
\[ T^\uparrow(f, \varepsilon) = \inf\left\{ t \geq 0: |f(t)| > \varepsilon \right\} \ \text{ and } \ \widetilde T^\uparrow(f, \varepsilon) = \inf\left\{ t \geq 0: |f(t)| \geq \varepsilon \text{ or } |f(t-)| \geq \varepsilon\right\}. \]

Recall that $\theta_t(f)= f(t + \, \cdot \,)$ and let $e_\varepsilon^\uparrow: D \to \Ecal$ be given by
\[ e^\uparrow_\varepsilon(f) = \theta_{T^\uparrow(e_\varepsilon(f), \varepsilon)}\left(e_\varepsilon(f)\right), \ f \in D. \]

We need preliminary results on continuity properties of the shift operator and also on the two times $T^\uparrow(f, \varepsilon)$ and $\widetilde T^\uparrow(f, \varepsilon)$.

\begin{lemma}\label{lemma:continuity-shift}
	If $f_n, f \in D$ and $t_n, t > 0$ are such that $f_n \to f$, $t_n \to t$ and $\Delta f_n(t_n) \to \Delta f(t)$, then $\theta_{t_n} (f_n) \to \theta_t(f)$.
\end{lemma}

\begin{proof}
	The proof is very similar to that of Lemma~\ref{lemma:continuity-concatenation}, so we only sketch it. Let $(\lambda_n) \in \Lambda_\infty$ such that $v_m(f_n \circ \lambda_n, f) \to 0$ for every $m \in M$, and assume without loss of generality that $\lambda_n(t) = t_n$. Considering $\nu_n(s) = \lambda_n(t+s)-t_n$, one can check that $(\nu_n) \in \Lambda_\infty$ and that $v_m(\theta_{t_n}(f_n) \circ \nu_n, \theta_t(f)) \to 0$ for every $m \in M$, which gives the result.
\end{proof}

For the following lemma remember that for $f = (f_1, \ldots, f_d): [0,\infty) \to \R^d$ we have $\norm{f}_t^d = \max_{1 \leq k \leq d} \norm{f_k}_t$.

\begin{lemma} \label{lemma:limsup}
	Let $f_n, f \in D$ and $\varepsilon > 0$. If $f_n \to f$, then $\limsup_n T^\uparrow(f_n, \varepsilon) \leq T^\uparrow(f, \varepsilon)$. If in addition $\widetilde T^\uparrow(f, \varepsilon) = T^\uparrow(f, \varepsilon)$, then $T^\uparrow(f_n, \varepsilon) \to T^\uparrow(f, \varepsilon)$.
\end{lemma}

\begin{proof}
	Note for simplicity $T^\uparrow = T^\uparrow(f, \varepsilon)$ and $T^\uparrow_n = T^\uparrow(f_n, \varepsilon)$. Consider $\eta > 0$ and $t_\eta \in (T^\uparrow, T^\uparrow + \eta)$ such that $f$ is continuous at $t_\eta$ and $|f(t_\eta)| > \varepsilon$. Such a $t_\eta$ always exists by definition of $T^\uparrow$, and because $f$ has at most countable many discontinuities. Then $|f_n(t_\eta)| \to |f(t_\eta)|$ and so $|f_n(t_\eta)| > \varepsilon$ for $n$ large enough, which entails for those $n$ that $T^\uparrow_n \leq t_\eta$. Hence $\limsup_n T^\uparrow_n \leq t_\eta \leq T^\uparrow + \eta$ and letting $\eta \to 0$ gives the result.
	
	Assume now that $\widetilde T^\uparrow(f, \varepsilon) = T^\uparrow(f, \varepsilon)$, and let $t < T^\uparrow$ be a continuity point of $f$, so that $\norm{f}_t^d < \varepsilon$. Then $\norm{f_n}_t^d \to \norm{f}_t^d$ and so $\norm{f_n}_t^d < \varepsilon$ for $n$ large enough. For those $n$ we therefore have $T^\uparrow_n > t$ and hence $\liminf_n T^\uparrow_n \geq t$. Since $t < T^\uparrow$ was arbitrary, letting $t \to T^\uparrow$ gives $T^\uparrow \leq \liminf_n T^\uparrow_n$.
\end{proof}

For $f \in D$ let $S_f \in D$ be the process recording its past supremum: $S_f(t) = \norm{f}^d_t$. Then $S_f$ is an increasing function and $T^\uparrow(f, \varepsilon) = \inf \left\{t \geq 0: S_f(t) > \varepsilon \right\} = S_f^{-1}(\varepsilon)$ with $S_f^{-1}$ the right-continuous inverse of $S_f$. If $f$ is continuous, then we also have
\[ \widetilde T^\uparrow(f, \varepsilon) = \inf \left\{ t \geq 0: |f(t)| \geq \varepsilon \right\} = \inf \left\{t \geq 0: S_f(t) \geq \varepsilon \right\} = \widetilde S_f^{-1}(\varepsilon) \]
with $\widetilde S_f^{-1}$ the left-continuous inverse of $S_f$. Since $\widetilde S_f^{-1}$ and $S_f^{-1}$ coincide exactly at continuity points of $S_f$ we get the following result.

\begin{lemma} \label{lemma:countable}
	If $f \in D$ is continuous, then
	\[ \left\{ \varepsilon > 0: T^\uparrow(f, \varepsilon) = \widetilde T^\uparrow(f, \varepsilon) \right\} = \left\{ \varepsilon > 0: \Delta S_f^{-1}(\varepsilon) = 0 \right\}. \]
\end{lemma}

\begin{prop} \label{prop:up}
	Consider $\varphi = \norm{\cdot}^d_\infty$. Assume that $X$ is continuous and that $e^\uparrow_\varepsilon(X_n) \Rightarrow e^\uparrow_\varepsilon(X)$ and $T(e^\uparrow_\varepsilon(X_n)) \Rightarrow T(e^\uparrow_\varepsilon(X))$ for every $\varepsilon$ such that $\Ncal(\varphi = \varepsilon) = 0$. Then it holds that $(e_\varepsilon, T \circ e_\varepsilon)(X_n) \Rightarrow (e_\varepsilon, T \circ e_\varepsilon)(X_n)$ for every $\varepsilon$ such that $\Ncal(\varphi = \varepsilon) = 0$.
\end{prop}

\begin{proof}
	In the rest of the proof let $\set = \{ \varepsilon > 0: \Ncal(\varphi = \varepsilon) = 0 \}$ and define $\Jcal(f) = (f, T(f))$ for $f \in D$. Using Lemma~\ref{lemma:joint-conv} the assumptions entail $\Jcal(e_\varepsilon^\uparrow(X_n)) \Rightarrow \Jcal(e_\varepsilon^\uparrow(X))$ for every $\varepsilon \in \set$. In the rest of this proof we fix $\varepsilon \in \set$ and note $E_n = e_\varepsilon(X_n)$ and $E = e_\varepsilon(X)$. The goal is to prove that $\Jcal(E_n) \Rightarrow \Jcal(E)$, and the proof operates in four steps: first we prove that $\Jcal(e_\delta^\uparrow(E_n)) \Rightarrow \Jcal(e_\delta^\uparrow(E))$ for any $0 < \delta < \varepsilon$ with $\delta \in \set$ and then that the sequence $(E_n, n \geq 1)$ is C-tight. Then we identify accumulation points and finally conclude that $\Jcal(E_n) \Rightarrow \Jcal(E)$.
	\\
	
	\setcounter{step}{1}
	\newstep Let $0 < \delta < \varepsilon$ with $\delta \in \set$: we prove that $\Jcal(e_\delta^\uparrow(E_n)) \Rightarrow \Jcal(e_\delta^\uparrow(E))$. So let $f: \Ecal \times [0,\infty) \to [0,\infty)$ be a bounded, continuous function: we must show that
	\[ \lim_{n \to +\infty} \E \left[ f \left( \Jcal\big(e_\delta^\uparrow(E_n)\big) \right) \right] = \E \left[ f \left( \Jcal\big(e_\delta^\uparrow(E) \big) \right) \right]. \]
	
	Since $X_n$ is regenerative, $E_n = e_\varepsilon(X_n)$ is equal in distribution to $e_\delta(X_n)$ conditioned on $\{\norm{e_\delta(X_n)}_\infty^d > \varepsilon\}$ and so
	\[ \E \left[ f \left( \Jcal\big(e_\delta^\uparrow(E_n) \big) \right) \right] = \E \left[ f \left( \Jcal\big(e_\delta^\uparrow(e_\delta(X_n)) \big) \right) \, \big | \, \norm{e_\delta(X_n)}_\infty^d > \varepsilon \right]. \]
	
	By definition we have $e_\delta^\uparrow \circ e_\delta = e_\delta^\uparrow$ and $\norm{e_\delta^\uparrow(f)}_\infty^d = \norm{e_\delta(f)}_\infty^d$, hence
	\[ \E \left[ f \left( \Jcal\big(e_\delta^\uparrow(E_n) \big) \right) \right] = \frac{\E \left[ f \left( \Jcal\big(e_\delta^\uparrow(X_n) \big) \right) \, ; \, \norm{e_\delta^\uparrow(X_n)}_\infty^d > \varepsilon \right]}{\P \left( \norm{e_\delta^\uparrow(X_n)}_\infty^d > \varepsilon \right)}. \]
	
	The continuous mapping theorem together with $\Jcal(e_\delta^\uparrow(X_n)) \Rightarrow \Jcal(e_\delta^\uparrow(X))$ imply that $(\Jcal \circ e_\delta^\uparrow, \varphi \circ e_\delta^\uparrow)(X_n) \Rightarrow ( \Jcal \circ e_\delta^\uparrow, \varphi\circ e_\delta^\uparrow)(X)$. Moreover, since $\Ncal(\varphi = \varepsilon) = 0$ we obtain
	\[ \P\left(\norm{e_\delta^\uparrow(X)}_\infty^d = \varepsilon\right) = \P\left(\norm{e_\delta(X)}_\infty^d = \varepsilon\right) = \Ncal \left( \varphi = \varepsilon \, | \, \varphi > \delta \right) = 0 \]
	and so letting $n\to +\infty$ we obtain
	\[ \lim_{n \to +\infty} \E \left[ f \left( \Jcal\big(e_\delta^\uparrow(E_n) \big) \right) \right] = \frac{\E \left[ f \left( \Jcal\big(e_\delta^\uparrow(X) \big) \right) \, ; \, \norm{e_\delta^\uparrow(X)}_\infty^d > \varepsilon \right]}{\P \left( \norm{e_\delta^\uparrow(X)}_\infty^d > \varepsilon \right)} = \E \left[ f \left( \Jcal\big(e_\delta^\uparrow(E) \big) \right) \right] \]
	invoking similar arguments for $X$ as for $X_n$ to get the last equality. This achieves the first step.
	\\
	
	\newstep We now prove the C-tightness of the sequence $(E_n)$. We control the oscillations, control over the supremum is given by similar arguments. Let $m, \zeta > 0$, we must show that
	\[ \lim_{\eta \to 0} \limsup_{n \to +\infty} \P \left( w_m(E_n, \eta) \geq \zeta \right) = 0. \]
	
	For any $t \geq 0$ and $\delta > 0$, we have by definition
	\[ E_n(t) = e_\delta^\uparrow(E_n) \big(t - T^\uparrow(E_n, \delta) \big) \]
	if $t \geq T^\uparrow(E_n, \delta)$ whereas $|E_n(t)| \leq \delta$ if $t < T^\uparrow(E_n, \delta)$. Hence for any $\delta, \eta > 0$ and $0 \leq s, t \leq m$ with $|t-s| \leq \eta$, we have $\left| E_n(t) - E_n(s) \right| \leq 2 \delta + w_m \big( e_\delta^\uparrow(E_n), \eta \big)$ (the difference $E_n(t) - E_n(s)$ is to be understood componentwise). Choosing any $\delta < \zeta / 2$ gives
	\[ \P \left( w_m(E_n, \eta) \geq \zeta \right) \leq \P \left( w_m(e^\uparrow_{\delta}(E_n), \eta) \geq \zeta / 2 \right) \]
	and in particular,
	\[ \lim_{\eta \to 0} \limsup_{n \to +\infty} \P \left( w_m(E_n, \eta) \geq \zeta \right) \leq \lim_{\eta \to 0} \limsup_{n \to +\infty} \P \left( w_m(e^\uparrow_{\delta}(E_n), \eta) \geq \zeta / 2 \right). \]
	
	For $\delta < \varepsilon$ we have $e^\uparrow_\delta(E_n) = e^\uparrow_\delta \circ e_\varepsilon (X_n) = e^\uparrow_\delta(X_n)$ which is by assumption $C$-tight for $\delta \in \set$. Thus for $\delta$ small enough (i.e., $\delta < \varepsilon$ and $\delta < \zeta/2$) and in $\set$, the right-hand side of the inequality in the previous display is zero, which ends this step.
	\\
	
	\newstep We prove that $E_n \Rightarrow E$: let $E'$ be any continuous accumulation point of the C-tight sequence $(E_n)$, and let $(u_n)$ be a subsequence such that $E_{u_n} \Rightarrow E'$. We must show that $E'$ is equal in distribution to $E$. Remember that $S_{E'}$ is the process of the past supremum of $E'$ with right-continuous inverse $S_{E'}^{-1}$, and let $H$ be the following deterministic set:
	\[ H = \left\{ \delta > 0: \P \left( \Delta S_{E'}^{-1}(\delta) = 0 \right) = 1 \right\}. \]
	
	Since $S_{E'}^{-1}$ is almost surely c\`adl\`ag it is well-known, see for instance Billingsley~\cite[Section~$13$]{Billingsley99:0}, that the set $H^c = [0,\infty) \setminus H$ is countable; thus $H$ is dense. Assume by Skorohod's representation theorem that the convergence $E_{u_n} \to E'$ holds almost surely. Let $\delta \in H$, and assume without loss of generality that $\Delta S_{E'}^{-1}(\delta) = 0$. Then Lemma~\ref{lemma:countable} gives $T^\uparrow(E', \delta) = \widetilde T^\uparrow(E', \delta)$, which combined with Lemma~\ref{lemma:limsup} implies that $T^\uparrow(E_{u_n}, \delta) \to T^\uparrow(E', \delta)$. In turn, this gives together with Lemma~\ref{lemma:continuity-shift} and the definition of $e^\uparrow_\delta$ that $e^\uparrow_\delta(E_{u_n}) \to e^\uparrow_\delta(E')$, since $E'$ is continuous.
	
	On the other hand we have proved that $e_{\delta}^\uparrow(E_{u_n}) \Rightarrow e_{\delta}^\uparrow(E)$ and so $e_{\delta}^\uparrow(E')$ and $e_{\delta}^\uparrow(E)$ are equal in distribution for every $\delta \in H$. Since $H$ is dense, we have $0 \in \overline H$ and so we can let $\delta \to 0$.
	
	Observe that $T^\uparrow(e, \delta) \to 0$ as $\delta \to 0$ if $e \in \Ecal$: if $|e(0)| > 0$ then $T^\uparrow(e, \delta) = 0$ for $\delta < |e(0)|$ while if $|e(0)| = 0$ then $T^\uparrow(e, \delta) \to 0$ follows by right-continuity of $e$ at $0$. As a consequence, $e_\delta^\uparrow(e)(t) \to e(t)$ as $\delta \to 0$ for any $e \in \Ecal$ and any $t \geq 0$, and we obtain, identifying final-dimensional distributions, that $E$ and $E'$ are equal in distribution. This proves that $E_n \Rightarrow E$. The previous arguments even show that $T^\uparrow(E_n, \delta) \Rightarrow T^\uparrow(E, \delta)$ for any $\delta \in H$, a fact that will be used in the next step.
	\\
	
	\newstep We now prove that $T(E_n) \Rightarrow T(E)$, from which the joint convergence $\Jcal(E_n) \Rightarrow \Jcal(E)$ follows similarly as for Lemma~\ref{lemma:joint-conv}. For any $\delta > 0$, we have by definition
	\[ T(E_n) = T^\uparrow(E_n, \delta) + T(e_\delta^\uparrow(E_n)) \]
	and so we have the following bounds, valid for any $\eta \in (0,1)$ and $x > 0$:
	\[ \P \left( T(e_\delta^\uparrow(E_n)) > x \right) \leq \P \left( T(E_n) \geq x \right) \leq \P\left( T^\uparrow(E_n, \delta) \geq \eta x \right) + \P \left( T(e_\delta^\uparrow(E_n)) \geq (1-\eta)x \right). \]
	
	Fix now $x$ such that $\P(T(E) = x) = 0$ and $\delta \in H$. Since $T^\uparrow(E_n, \delta) \Rightarrow T^\uparrow(E, \delta)$ by the previous step and $T(e_\delta^\uparrow(E_n)) \Rightarrow T(e_\delta^\uparrow(E))$ by the first step, the portmanteau theorem gives
	\begin{multline*}
		\P \left( T(e_\delta^\uparrow(E)) > x \right) \leq \liminf_{n \to +\infty} \, \P \left( T(E_n) \geq x \right) \leq \limsup_{n \to +\infty} \, \P \left( T(E_n) \geq x \right) \leq \P\left( T^\uparrow(E, \delta) \geq \eta x \right)\\
		+ \P \left( T(e_\delta^\uparrow(E)) \geq (1-\eta)x \right).
	\end{multline*}
	
	Recall that $T^\uparrow(e, \delta) \to 0$ as $\delta \to 0$. Since $T(e) = T^\uparrow(e, \delta) + T(e_\delta^\uparrow(e))$ this shows that $T(e_\delta^\uparrow(e)) \to T(e)$. Thus letting $\delta \to 0$ in the previous display gives 
	\[
		\P \left( T(E) > x \right) \leq \liminf_{n \to +\infty} \P \left( T(E_n) \geq x \right) \leq \limsup_{n \to +\infty} \P \left( T(E_n) \geq x \right) \leq \P \left( T(E) \geq (1-\eta)x \right).
	\]
	
	Letting now $\eta \to 0$ ends to prove that $T(E_n) \Rightarrow T(E)$ since $x$ was chosen such that $\P(T(E) = x) = 0$. The proof is complete.
\end{proof}

Combining Theorem~\ref{thm:main-2} and Proposition~\ref{prop:up} we obtain the following result.

\begin{theorem}
\label{thm:main-3}
	Let $V = \R^d$, $a = 0$ and $\varphi = \norm{\cdot}_\infty^d$. Let $\set \subset (0,\infty)$ such that $0 \in \overline \set$ and $\Ncal(\varphi=\varepsilon)=0$ for all $\varepsilon \in \set$. If $X$ is continuous, $g_\varepsilon(X_n) \Rightarrow g_\varepsilon(X)$, $e^\uparrow_\varepsilon(X_n) \Rightarrow e^\uparrow_\varepsilon(X)$ and $T(e^\uparrow_\varepsilon(X_n)) \Rightarrow T(e^\uparrow_\varepsilon(X))$ for every $\varepsilon \in \set$, then $X_n \Rightarrow X$.
\end{theorem}

\section{Generalization of Theorem~\ref{thm:main}} \label{sec:discussion}

In this section we do not assume anymore that $X_n$ and $X$ are regenerative processes in the sense of~\eqref{eq:def-reg}. The deterministic results of Section~\ref{sub:deterministic} concern concatenation of paths that are not assumed to be excursions; this allows to generalize Theorem~\ref{thm:main} to such a general setting. Let us begin by recalling It\^o's construction~\cite{Ito72:0}, which starts from a $\sigma$-finite measure $\Ncal$ on $\Ecal \setminus \{a\}$ and builds a process $X$ with excursion measure $\Ncal$.

Assume that $\Ncal$ satisfies $\Ncal(1 \wedge T) < +\infty$, let $\partial$ be a cemetery point and $(\alpha_t, t \geq 0)$ be a $\{\partial\} \cup (\Ecal \setminus \{\zero\})$-valued Poisson point process with intensity measure $\Ncal$. Let $d \geq 0$ and
\[ Y(t) = dt+\sum_{0 \leq s \leq t} T(\alpha_s) \]
with the convention $T(\partial) = 0$. Since $\Ncal(1 \wedge T) < +\infty$, $Y$ is well-defined and is a subordinator with drift $d$ and L\'evy measure $\Ncal(T \in \, \cdot \,)$. Let $Y^{-1}$ be its right-continuous inverse and define the process $X$ as follows:
\[ X(t) = \alpha_{Y^{-1}(t-)} \left( t - Y^{-1}(Y(t)-) \right) \]
for $t \geq 0$ such that $\Delta Y(Y^{-1}(t)) > 0$ and $0$ otherwise.

The key observation is that this construction also works if $\Ncal$ is a $\sigma$-finite measure on the set of paths $f$ killed at some time $\kappa(f) \in (0,\infty]$, which may not be the first hitting time of $a$: then one needs to consider the subordinator with L\'evy measure $\Ncal(\kappa \in \, \cdot \,)$ instead of $\Ncal(T \in \, \cdot \,)$. Let us call this construction the \emph{extended It\^o's construction}. The main difference with the setting of Theorem~\ref{thm:main} is that it may not be possible anymore to recognize the regenerative motifs on a sample path of $X$, i.e., the generalization of the map $e_\varepsilon$ to this setting may not be well-defined. Consider for instance the case where the motif is the concatenation of a random number of excursions. Hence the generalization of Theorem~\ref{thm:main} takes a slightly different form.

In the following statement, $I: D \to D$ is the identity map and $\varphi$ is now a map with domain $D$, i.e., $\varphi: D \to [0,\infty]$. Moreover, for $X$ obtained by the extended It\^o's construction from a measure $\Ncal$ and a drift coefficient $d$, we define $g^e_\varepsilon(X)$ as the left endpoint of the first motif $m$ with $\varphi(m) > \varepsilon$. The notation $g^e_\varepsilon(X)$ is a little abusive, since rigorously $g^e_\varepsilon$ is well-defined as a function of the Poisson point process of motifs $\alpha$ of $X$, but not necessarily as a function of $X$.

\begin{theorem}
	Let $\Ncal$ and $\Ncal_n$ be $\sigma$-finite measures on the set of killed paths. Assume that $\Ncal$ has infinite mass and let $X$ and $X_n$ the processes obtained by applying the extended It\^o's construction to $\Ncal$ and $\Ncal_n$, respectively (and arbitrary drift coefficients). Let $\set \subset (0,+\infty)$ such that $0 \in \overline \set$ and $\Ncal(\varphi =\varepsilon)=0$ for all $\varepsilon \in \set$. Assume that the sequence $(X_n)$ is tight and that $g_\varepsilon^e(X_n) \Rightarrow g_\varepsilon^e(X)$ and that $(I, \kappa, \varphi)$ under $\Ncal_n(\, \cdot \, | \, \varphi > \varepsilon)$ converges to $(I, \kappa, \varphi)$ under $\Ncal(\, \cdot \, | \, \varphi > \varepsilon)$ for every $\varepsilon \in C$. Then $X_n \Rightarrow X$.
\end{theorem}

Let us conclude by one remark and one example that motivate this extension. The remark is that the classical definition~\eqref{eq:def-reg} of a regenerative process is, to some extent, quite restrictive. For instance, this definition excludes processes that stay at $a$ for a duration that is not exponential. It also excludes processes (even Markovian ones) that have non-zero holding times but leave $a$ continuously. Consider for instance the Markov process that stays at $0$ for an exponential time, then increases at rate $1$ for an exponential time and jumps back to $0$: then~\eqref{eq:def-reg} fails for $\tau = \inf\{ t \geq 0: X(t) > 0 \}$.

The second example comes from queueing theory: consider the example of the $G/G/1$ queue. Then the queue length process does not regenerate when it hits $0$, because when it hits $0$ it stays there for a duration that may depend on the excursion that just finished. On the other hand, it does regenerate when a customer initiates a busy cycle, i.e., when the queue length process jumps from $0$ to $1$.

Although not regenerative in the sense of~\eqref{eq:def-reg}, all these processes can be obtained via the previous extended It\^o's construction and therefore fall within the framework of the previous theorem.

\end{document}